\documentclass{article}
\usepackage{amsmath,amssymb,amsthm}
\usepackage[svgnames, table]{xcolor} 
\usepackage[ruled,vlined]{algorithm2e}
\usepackage{setspace}
\usepackage{xr-hyper}
\usepackage[colorlinks=true, linkcolor=blue, citecolor=blue, filecolor=blue, runcolor=blue, urlcolor = blue]{hyperref}

\usepackage{fullpage}
\newtheorem{definition}{Definition}[section]
\newtheorem{theorem}{Theorem}[section]
\newtheorem{proposition}{Proposition}[section]

\newtheorem{corollary}{Corollary}[section]
\newtheorem{remark}{Remark}[section]
\newtheorem{example}{Example}[section]

\usepackage{graphicx,calc}
\DeclareGraphicsExtensions{.png}

\begin{document}

\title{The Jones polynomial in systems with Periodic Boundary Conditions}

\author{Kasturi Barkataki \and Eleni Panagiotou}

\date{\today}

\maketitle

\begin{abstract}
Entanglement of collections of  filaments arises in many contexts, such as in polymer melts, textiles and crystals. Such systems are modeled using periodic boundary conditions (PBC), which create an infinite periodic system whose global entanglement may be impossible to capture and is repetitive.  We introduce two new methods to assess topological entanglement in PBC: the Periodic Jones polynomial and the Cell Jones polynomial. These tools capture the grain of entanglement in a periodic system of open or closed chains, by using a finite link as a representative of the global system. These polynomials are topological invariants in some cases, but in general are sensitive to both the topology and the geometry of physical systems. For a general system of 1 closed chain in 1 PBC, we prove that the Periodic Jones polynomial is a recurring factor, up to a remainder, of the Jones polynomial of a conveniently chosen finite cutoff of arbitrary size of the infinite periodic system. We apply the Cell Jones polynomial and the Periodic Jones polynomial to physical PBC systems such as 3D realizations of textile motifs and polymer melts of linear chains obtained from molecular dynamics simulations.  Our results demonstrate that the Cell Jones polynomial and the Periodic Jones polynomial can measure collective entanglement complexity in such systems of physical relevance. 
\vspace{2pc}


\end{abstract}
%
%
%
\section{Introduction}

Many physical systems, such as polymers, textiles, and crystals are composed of filamentous structures, whose entanglement complexity largely determines their mechanical properties and function \cite{Arsuaga2005,Edwards1967,Liu2018,Panagiotou2019,Qin2011,Sulkowska2012,Taylor1974,Li2019,Zhang2022}. The constituent filaments in these systems can be of varying architecture, such as ring or linear, which can be represented by closed or open curves in 3-space, respectively. Even though entanglement of closed curves in 3-space is well defined in knot theory, the necessary framework to rigorously measure entanglement of linear chains (without any approximation schemes) was discovered only recently 
\cite{Panagiotou2015,Panagiotou2021,Panagiotou2020b,Wang2022,Panagiotou2011}.  In practice however, many physically relevant systems of multiple filaments in 3-space are modeled using Periodic Boundary Conditions (PBC). Measuring the entanglement in systems employing PBC is considerably more complex, since the generated system is infinite. In this paper, we introduce the Periodic Jones polynomial that can measure entanglement in systems of open and/or closed chains in PBC.

It is natural to look for measures of entanglement of curves in three-dimensional space in the theory of knots and links \cite{knot-book,Kauffman2001}. A knot (or link) is one (or more) simple closed curve(s) in space. Knots
and links are classified with respect to their complexity by topological invariants, usually of the form of integer valued functions or polynomials with integer coefficients \cite{Freyd1985,Kauffman1990,Przytycki1987}. These measures are invariant under continuous deformations of the chains that do not allow self intersections. Although open curves are not knotted or linked in the topological sense, they can form complex conformations, which we call entangled. Entanglement of open curves can also  be measured rigorously by topological/geometrical measures, which are either real numbers or polynomials with real coefficients, that are continuous functions of the curve coordinates and can detect knotting and threading \cite{Barkataki2022, Gauss1877, Panagiotou2020b, Panagiotou2021}.

Periodic boundary conditions are often employed to model physical systems of filaments in order to avoid boundary effects. The entanglement in such systems has several characteristics that make it more difficult to quantify \cite{Morton2009, Qin2011, Fukuda2023, Fukuda2022,Kolbe2022, Panagiotou2015, Castle2011, Delgado2017, Evans2013b, Evans2015a, Evans2013a, Evans2015b, Markande2020, Knittel2020, Wadekar2020, Wadekar2021, OKeeffe2021, OKeeffe2022, rosi2005rod}. A system with PBC is created by infinite copies of a base
cell which creates an infinite system whose collective entanglement is impossible to compute as a whole.  One may focus on the entanglement of the arcs present inside one single cell to assess the entanglement of the periodic system, but this poses several difficulties. First, the arcs composing a cell are open mathematical curves and measuring the multi-chain complexity of a collection of open curves in 3-space became possible only recently \cite{Barkataki2022}. Second, focusing only on one cell can lead to missing important topological information that can be seen only by accounting for its periodic translations which may create longer connected components with higher topological complexity. To address the periodicity, prior studies have proposed to work instead in a topological identification space (for example a solid torus, thickened torus or the 3-torus). However, these methods are not sensitive on the geometry of a system since they rely on the notion of topological equivalence (which allows deformations) and are thus are not well defined for open (finite linear) curves. Moreover,  even in the case of closed or infinite filaments, the geometry may be of interest \cite{Wadekar2020,Wadekar2021}. The periodic linking number was introduced in \cite{Panagiotou2015}, as a method to capture pairwise entanglement in periodic systems of chains of both linear and ring architecture, and it is successfully applied to periodic systems to analyze their pairwise entanglement \cite{Panagiotou2014,Panagiotou2011,P2018,Millett2016,Panagiotou2013,Panagiotou2019,Igram2016}. The periodic linking number is successful in measuring pairwise entanglement in PBC, of both closed and open curves, because, it was based on defining linking directly in the continuum periodic system, without reference to the base space. Extending this approach to other measures of higher order topological entanglement, like the Jones polynomial, is not obvious, since it is difficult to decouple or estimate the contributions of any single finite link in the system to the Jones polynomial of the whole system. So far, it has been possible to define the Jones polynomial for systems of closed curves in 1 and 2 PBC using identification spaces \cite{Qin2011,Morton2009}. In this manuscript, we define the Jones polynomial for systems of open (and closed) chains in PBC for the first time. We provide two definitions of the Jones polynomial of the periodic system - the Cell Jones polynomial, which captures the multi-chain complexity of arcs in a base cell using \cite{Barkataki2022}, and the Periodic Jones polynomial, which captures the grain of the global complexity of the infinite system. These polynomials apply to both open and closed curves in 1, 2 or 3 PBC. The definition of the Periodic Jones polynomial relies on extracting a finite grain of entanglement from the infinite periodic system that is minimal but accounts for the multi-chain complexity that is periodically repeated, while taking into account entire unfoldings of chains instead of fragments of chains. The Periodic Jones polynomial is related to the classical Jones polynomial of the infinite system, in the only meaningful sense, that is, the Jones polynomial of a cutoff of arbitrary size. More precisely, for a general system of 1 closed chain in 1 PBC,  we prove that the Periodic Jones polynomial is a recurring factor, up to a remainder, of the Jones polynomial of any $N^{th}$ (appropriately chosen) cutoff of the infinite system for any $N$. We prove that the Periodic Jones polynomial and the Periodic Linking Number are sufficient to fully describe the contribution of a particular state of the Jones polynomial of the $N^{th}$ cutoff of the infinite system for any $N$. The use of these new polynomials is demonstrated by applying them to polymeric systems of linear chains, obtained from Molecular Dynamics Simulations, of varying molecular weight and to doubly-periodic systems such as 3D realizations of textile motifs. Our results show that these tools can classify and compare among the complexity of different systems.

The paper is organized as follows: Section \ref{sec_prelim_defs} presents the definitions of measures of entanglement of collections of curves in 3-space, that are useful in this study. Section \ref{sec_PBC} presents necessary definitions for studying entanglement in PBC and introduces the Cell Jones polynomial and the Periodic Jones polynomial and discusses their properties. Section \ref{sec_app} presents results on the application of the Cell and Periodic Jones polynomials to multi-chain systems in PBC such as, 3D realizations of textile motifs and  polymer melts of linear chains. Section \ref{sec_mpl_inf} discusses the relation of the Periodic Jones polynomial and the Jones polynomial of the  infinite periodic system. Finally, Section \ref{sec_conc} presents the conclusions of this study.
\section{The Jones polynomial of open and closed curves in 3-space}\label{sec_prelim_defs}

In this section, we discuss the Jones polynomial for open and closed curves in 3-space which was introduced in \cite{Barkataki2022,Panagiotou2020}. The Jones polynomial of a collection of open curves, $l$, is defined with respect to projections of $l$, which we call diagrams, and it is the average of the Jones polynomial in a projection over all projections.

\begin{definition}(Jones Polynomial of open or closed chains in 3-space) Let $l$ denote a collection of open or closed oriented curves in 3-space. The Jones polynomial of $l$ is defined as,
\begin{equation}\displaystyle
f_{K}(l)=\frac{1}{4\pi}\int_{\vec{\xi}\in S^2}\left(-A^3\right)^{-\mathsf{Wr}\left(l_{\vec{\xi}}\right)}\left \langle (l)_{\vec{\xi}}\right \rangle dS.
\label{Jones_OC}
\end{equation}

\noindent where the integral is over all vectors in $S^2$ except a set of measure zero (corresponding to non-generic projections), where  $ (l)_{\vec{\xi}}$ is a diagram obtained from the projection of $l$ on the plane with normal vector ${\vec{\xi}}$, whose Writhe is given by $\mathsf{Wr}(l_{\vec{\xi}})$ (defined below) and bracket polynomial is given by $\left\langle (l)_{\vec{\xi}}\right\rangle$ (defined below) \cite{Barkataki2022}.
\label{jones-open-def}
\end{definition}

 Notice that $l_{\vec{\xi}}$ can be a closed or open knot or link diagram. When it is an open knot or link, it is called a knotoid or linkoid, respectively. The Writhe of a diagram and the bracket polynomial for a diagram are defined as follows:

\begin{definition}(Writhe)
Let $(l)_{\vec{\xi}}$ denote the diagram of a collection of oriented open curves. 
Each crossing in a projection is of one of the two types shown in Figure \ref{diagcross1}, associated with a positive and a negative sign, respectively. The \textit{writhe} of $(l)_{\vec{\xi}}$ is defined as the algebraic sum of crossings in the diagram, namely:
\begin{equation}
\mathsf{Wr}(l_{\vec{\xi}})=\sum_{\text{crossings in diagram}}\text{sign of crossing}.
\end{equation}
\end{definition}

\begin{figure}[ht!]
    \centering
    \includegraphics[scale=0.15]{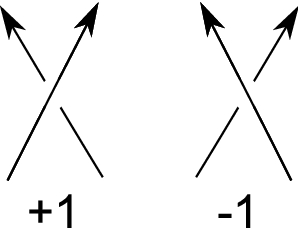}
    \caption{The possible signs at any crossing in an oriented diagram of a knot/link or a knotoid/linkoid (open knot/link diagrams).}
    \label{diagcross1}
\end{figure}

\begin{remark}
    In Section \ref{sec_mpl_inf} we will also refer to the diagrammatic (periodic \cite{Panagiotou2015}) linking number between links as the half algebraic sum of inter (shared) crossings between two links. 
\end{remark}

\begin{definition}(Bracket polynomial of an open link diagram (linkoid))
Let $(l)_{\vec{\xi}}$ denote the diagram of $n$ components. The bracket polynomial is completely characterised by the following Skein relation and initial conditions:
\begin{equation}
\begin{split}
\left\langle\raisebox{-9pt}{\includegraphics[width=.05\linewidth]{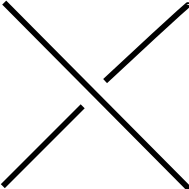}}\right\rangle=A\left\langle\raisebox{-9pt}{\includegraphics[width=.05\linewidth]{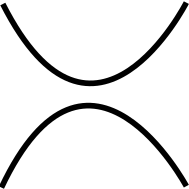}}\right\rangle+A^{-1}\left\langle\raisebox{-9pt}{\includegraphics[width=.05\linewidth]{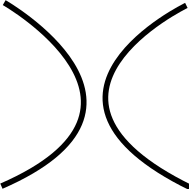}}\right\rangle, & \hspace{0.5cm}\left\langle L\cup \bigcirc\right\rangle=d\left\langle L\right\rangle,\hspace{0.5cm}\left\langle \raisebox{-12pt}{\includegraphics[width=.12\linewidth]{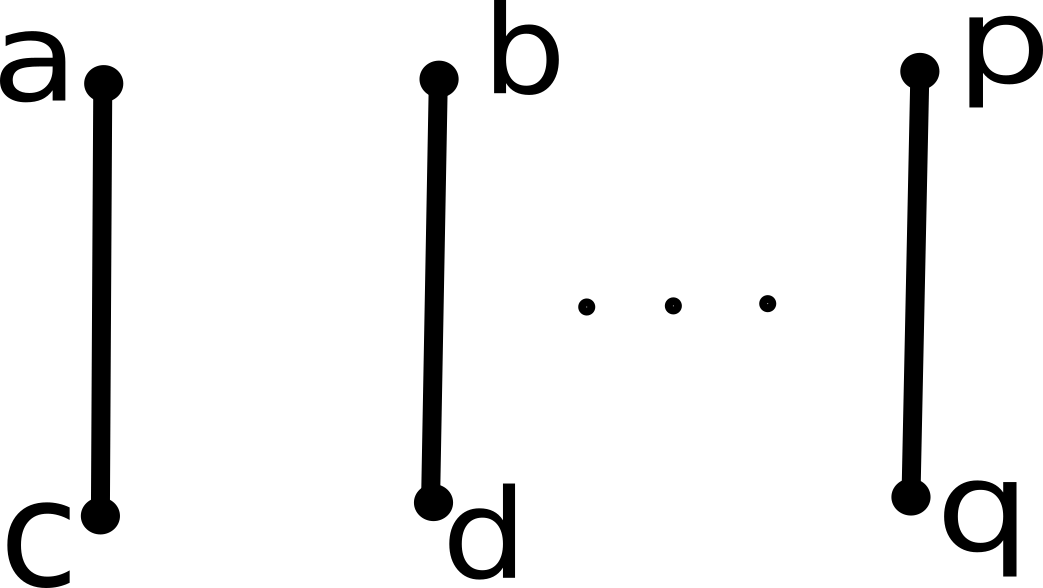}}\right\rangle=d^{|cyc|},
\label{bkt_sk_L}
\end{split}
\end{equation}
where $d=-A^2-A^{-2}$, $L$ denotes any linkoid, $|cyc|$ denotes the number of distinct segment cycles corresponding to a final state of the bracket state sum expansion \cite{Barkataki2022}, and where $A=t^{-1/4}$. For closed curves, where the bracket in the last equation contains $n$ disjoint circles, $|cyc|=n-1$. 
\end{definition}
The Jones polynomial of collections of open or closed curves in 3-space has the following properties \cite{Barkataki2022}:
\begin{enumerate}
    \item It does not depend on any particular direction of projection of the collection of open or closed curves.
    \item For a collection of open curves in 3-space, it is not the same as the Jones polynomial of a corresponding/approximating
link, nor that of a corresponding/approximating linkoid.
    \item For a collection of closed curves in $3$-space, the Jones polynomial in Equation \ref{Jones_OC} is same as the classical Jones polynomial and it can be computed from a single projection, i.e. $ f_{\mathcal{L}}=f_{\mathcal{L}_{\vec{\xi}}}$ where, $\vec{\xi} \in S^2$ is any projection vector.
    \item The Jones polynomial of a collection of open curves in 3-space has real coefficients. It is not a topological invariant, but it is a continuous function of the curve coordinates.
    \item As the endpoints of a collection of open curves in 3-space tend to coincide for each component, the Jones polynomial tends to that of the corresponding link.
\end{enumerate}
\section{Entanglement in PBC systems and the Periodic Jones polynomial}
\label{sec_PBC}

Systems in PBC may comprise of open, closed or infinite chains (see Figure \ref{3_pbc_illus}). In Section \ref{PBC_termino}, we describe systems employing PBC along with some useful terminologies. In Section \ref{Jones_cell}, the Cell Jones polynomial is introduced to study the local complexity of a PBC system. In Section \ref{Jones_periodic}, the Periodic Jones polynomial is introduced as a measure of global entanglement of the PBC system. Lastly, in Section \ref{modulo}, a normalization scheme for the Jones polynomial is discussed in the context of systems employing PBC.
\begin{figure}[ht!]
    \centering
    \includegraphics[scale=0.03]{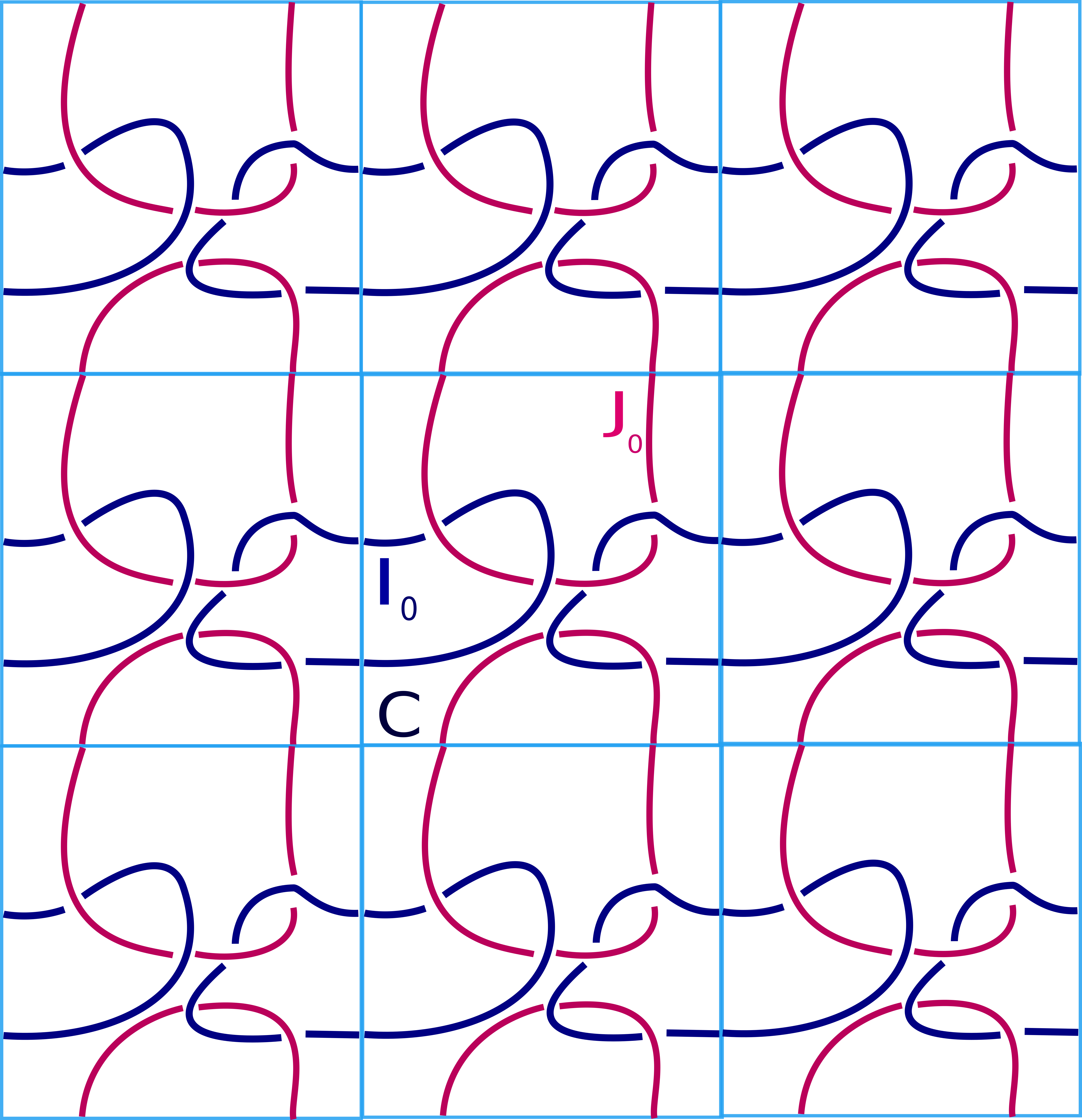} \quad
    \includegraphics[scale=0.03]{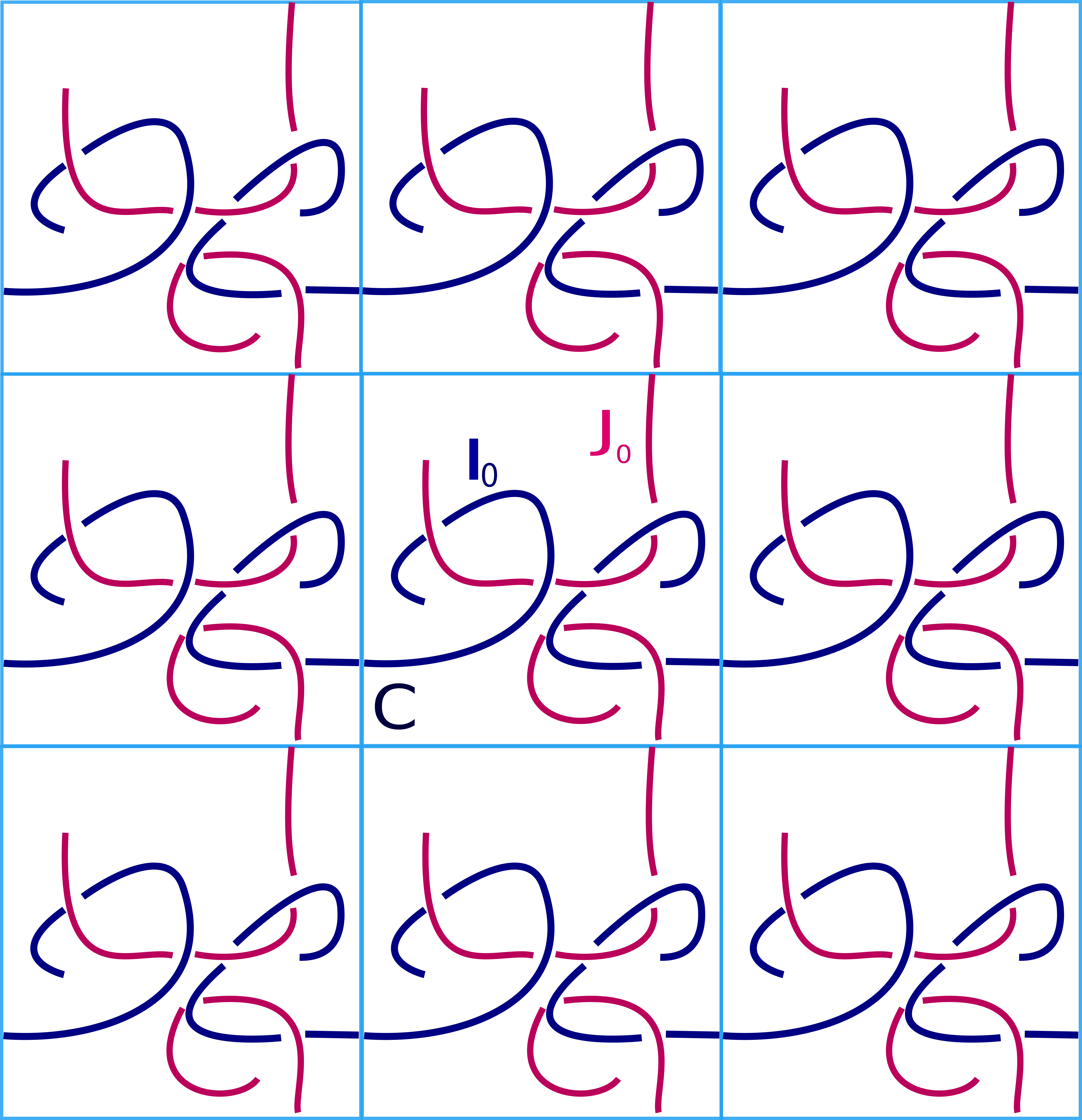} \quad
    \includegraphics[scale=0.03]{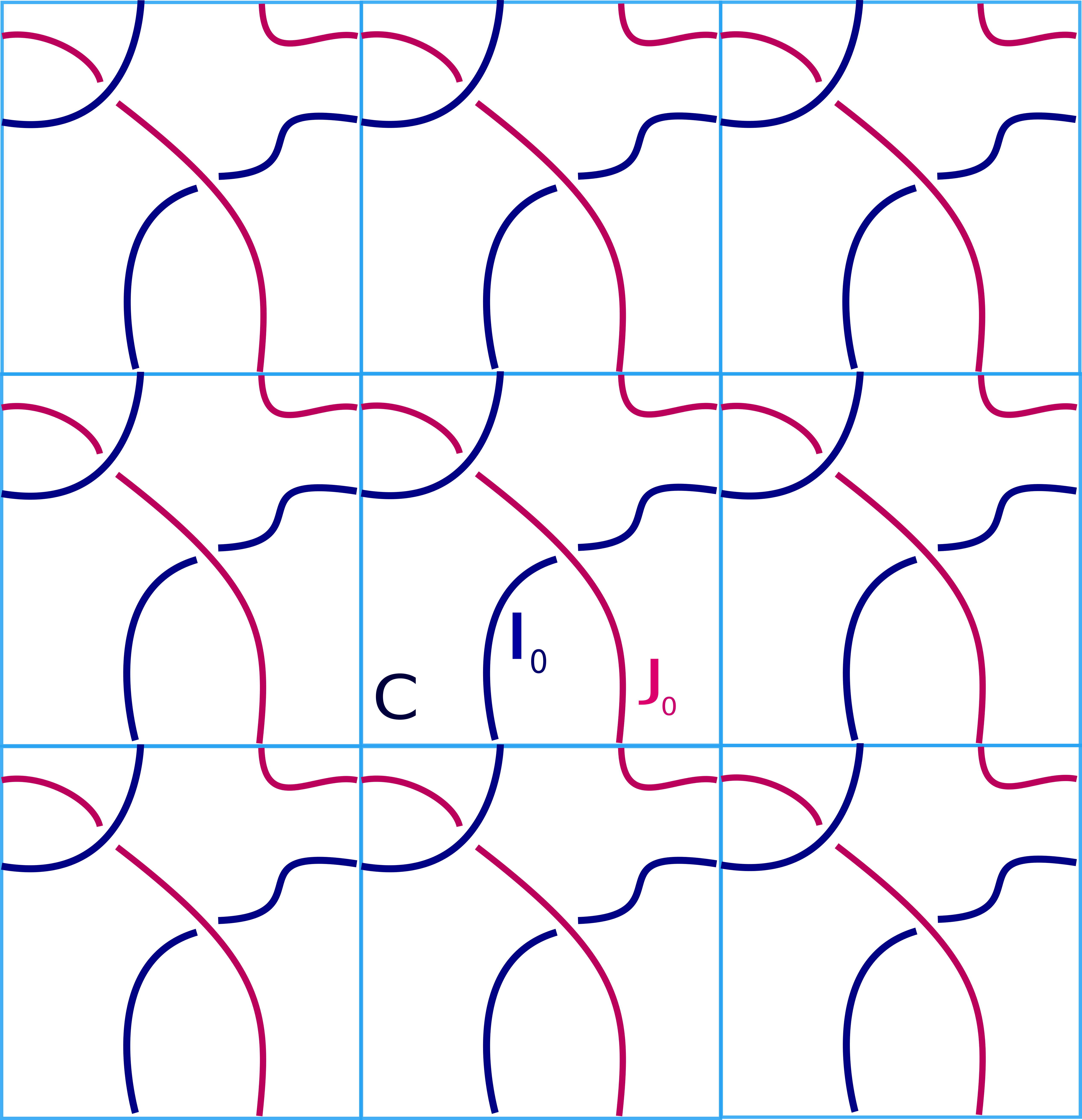}
    \caption{Left : A system in 2 PBC with 2 closed generating chains. Center : A system in 2 PBC with 2 open generating chains. Right :  A system in 2 PBC with 2 infinite generating chains.}
    \label{3_pbc_illus}
\end{figure}

\subsection{Systems with periodic boundary conditions}\label{PBC_termino}
The following definitions are similar to those in \cite{Panagiotou2015} and they help study entanglement in PBC.
\begin{definition}
A cell consists of a cube with embedded arcs (i.e. parts of curves) which satisfy the PBC requirement. We define a system employing PBC as a 3D system generated by tiling identical cubic cells (or parallelepipeds) of volume $xyz$, where $x,y$ and $z$ denote length, breadth and height of the cell, respectively. A generating chain, say $i$, is the union of all the arcs inside the generating cell, the translations of which define a connected component in the periodic system. A base point is defined for every generating chain and it is a point in one of its arcs. The corresponding free chain, denoted $I$, is defined as the collection of all translations of the generating chain $i$. When $I$ consists of a collection of closed/open/infinite curves, it is called a closed/open/infinite free chain, respectively. An image of the free chain $I$, we denote $I_k$, is any connected component in the periodic system that consists of a single translation of each of the arcs of the generating chain $i$. Therefore, each image contains only one base point. The image, $I_0$, of a free chain $I$, whose base point lies in the generating cell is called the parent image. The smallest union of the copies of the unit cell needed for one unfolding of a generating chain, i.e. for completion of one of its images, $I_0$, is the minimal unfolding and it is denoted by $\operatorname{mu}(I_0)$.
\end{definition}

\begin{remark}\label{infbase}
    Notice that even if $I$ is an infinite free chain, then an image of $I$ is still a finite length arc in the periodic system (since an image cannot have multiple copies of any generating arc). For closed or open free chains, an image does not depend on the choice of basepoint in the generating chain. However, in the case of infinite free chains, an image does depend on the choice of basepoint (a similar ambiguity is discussed in \cite{Fukuda2022, Fukuda2023, Morton2009}). Similarly, the minimal unfolding of an image of a closed or open free chain is independent of a choice of basepoint for the system, but in the case of infinite free chains it does depend on it.
\end{remark}

\begin{figure}[ht!]
    \centering
    \includegraphics[scale=0.03]{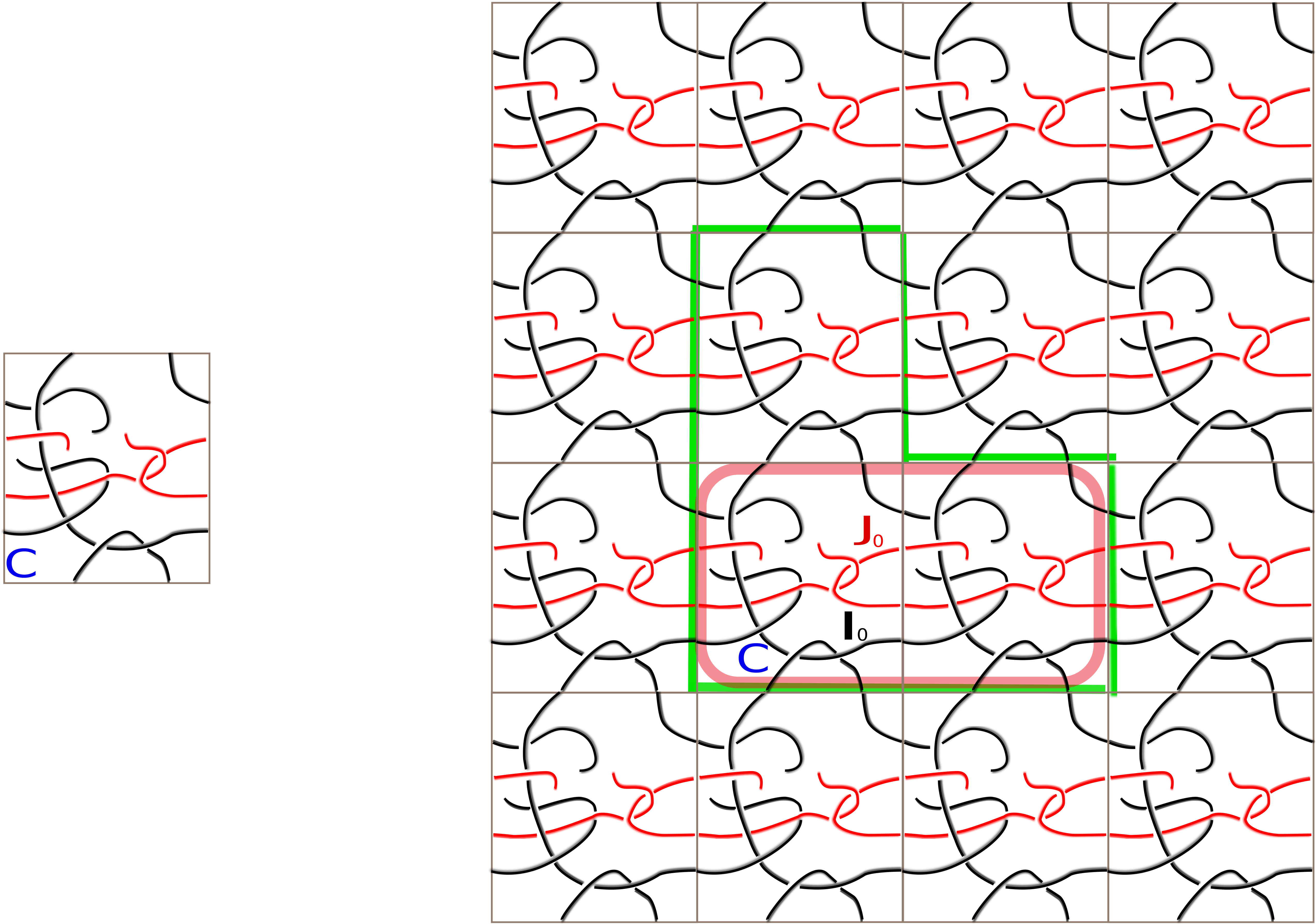}
    \caption{2D cartoon representation of a 3D system employing 2 PBC. Left: The base cell, $C$, which generates a 2 PBC system. Right: The free chain, $I$, (resp. $J$ ) is the set of  black (resp. red) chains in the periodic system. $I_0$ (resp. $J_0$), is the parent image of the free chain, $I$ (resp. $J$) and its minimal unfolding, $\operatorname{mu}(I_0)$, (resp. $\operatorname{mu}(J_0)$) is highlighted in green (resp. red)}
    \label{PBCterms}
\end{figure} 
\noindent See Figure \ref{PBCterms} for an illustrative example of two free chains in a system employing 2 PBC. 

\subsection{The Cell Jones polynomial}\label{Jones_cell}

A generating cell consists of a collection of arcs that can be seen as a collection of curves in 3-space whose collective entanglement can be measured using the method introduced in \cite{Barkataki2022}.
Given a PBC system and its generating cell, the Cell Jones polynomial is defined as follows:
\begin{definition}(Cell Jones polynomial)
Let $l_C$ denote the collection of generating arcs in a generating cell, $C$, of a periodic system. The Cell Jones polynomial, $V_C$, is defined as the Jones polynomial of $l_C$, namely,
$\displaystyle \mathsf{V}_C=\mathsf{V}(l_C)$.
\label{proper_CJP}
\end{definition}

The Cell Jones polynomial captures the topological/geometrical complexity of the collections of arcs lying inside a unit cell of the PBC system and has the following properties:
 
\begin{enumerate}
    \item When the images of the free chains do not touch the faces of the cell, the Cell Jones polynomial is equal to the Jones polynomial of the closed/open link inside a cell.
    \item It is preserved under any periodic translation of the unit cell with respect to the frame of reference.
    \item Notice that, $\mathsf{V}_C$ treats generating arcs that belong to the same generating chain as different components and these can be open curves in 3-space. Thus, $\mathsf{V}_C$ is a polynomial of real coefficients and not a topological invariant, even for generating chains that give rise to closed chains in the periodic system. It is a continuous function of the curve coordinates,  under deformations that do not change the number of components that intersect the cell.
\end{enumerate}

\subsection{The Periodic Jones polynomial}\label{Jones_periodic}

If we only consider the entanglement with respect to a base cell, as described in the previous section, we may be unable to detect important topological information that can be seen only by accounting the entanglement present in the periodic system. In particular, the link corresponding to a single cell may fail to capture the total topological complexity of any image of any free chain and the total topological entanglement imposed on an image of any of the free chains by the rest. In this section, we provide a means to capture the global topological complexity of a PBC system via a finite collection of curves, whose Jones polynomial satisfies the following requirements :
\begin{enumerate}
    \item It is a finite polynomial in one variable.
    \item When the images of the free chains do not touch the faces of the cell, it is equal to the Jones polynomial of the link inside a cell.
    \item It captures the topological complexity of any image of any free chain.
    \item It captures the total topological complexity imposed on an image of each free chain in the system by all the other chains.
    
\end{enumerate}
The Periodic Jones polynomial is defined by using as reference a larger cell, called the minimal collective unfolding, which is the minimal, convex collection of cells needed to unfold an entire image of each component, namely:

\begin{definition}(Minimal collective unfolding) Let $C$ be the base cell of a PBC system in 3-space, with $n$ generating chains. Let $I_0^{(i)}$ denote the parent image of the $i^{th}$ chain, where $i \in \{ 1, 2, \cdots n \}$ and let $\operatorname{mu}(I_0^{(i)})$ denote its minimal unfolding of size, $x_{\operatorname{mu}\left(I_0^{(i)}\right)} \times y_{\operatorname{mu}\left(I_0^{(i)}\right)} \times z_{\operatorname{mu}\left(I_0^{(i)}\right)}$. The minimal collective unfolding, $\mathcal{MU}_C$, is defined as the parallelepiped with size $x_{\mathcal{MU}_C}\times y_{\mathcal{MU}_C}\times z_{\mathcal{MU}_C}$, where $\displaystyle x_{\mathcal{MU_C}}=\operatorname{max}  \{ x_{\operatorname{mu}(I_0^{(i)})} \}_{i=1}^n \quad , \quad y_{\mathcal{MU}_C}=max\{ y_{\operatorname{mu}(I_0^{(i)})} \}_{i=1}^n \quad , \quad z_{\mathcal{MU}_C}=max\{ z_{\operatorname{mu}(I_0^{(i)})} \}_{i=1}^n.$ 
\label{MCU}
\end{definition}

\begin{remark}
 Given a periodic system and a generating cell, $C$,  $\mathcal{MU}_C$ is unique and it is independent of its location in the periodic system. It is a finite sheeted covering of the identification space of the generating cell. 
 \end{remark} 
\begin{proposition}
The minimal collective unfolding, $\mathcal{MU}_C$, is the smallest convex set formed by a union of cells of the periodic system such that it contains at least one image of each free chain. 
\label{mu_smallest}
\end{proposition}

\begin{proof}
Suppose $x \times y \times z$ are the dimensions of the unit cell $C$ and that $n_1 x \times n_2 y \times n_3 z$ are the dimensions of $\mathcal{MU}_C$. Suppose, without loss of generality, that $n_1=x_{\operatorname{mu}(I_0^{(i)})}$ for some $i$ generating chain $i$. Then any other convex set of size $(n_1-1) x \times n_2 y \times n_3 z$ cannot contain an image of the free chain $I$. 
\end{proof}

\begin{definition}(Minimal periodic link) Let $C$ be the base cell of a PBC system in 3-space, generated by $n$ free chains. The minimal periodic link, $\mathcal{L}_C$, is defined to be the link consisting of all the images of the $n$ free chains which intersect or are contained in the minimal collective unfolding, $\mathcal{MU}_C$.
\label{FER}
\end{definition}

Figure \ref{mu1} illustrates an example of the minimal collective unfolding, $\mathcal{MU}_C$, and the corresponding minimal periodic link, $\mathcal{L}_C$, for the 2 PBC system in Figure \ref{PBCterms}. Given the coordinates of the parent image for each generating chain in the system, the minimal collective unfolding and the minimal periodic link can be constructed via Algorithm \ref{AlgoH}.   
\begin{algorithm}[ht!]
\SetAlgoLined
Set the values $x_{\mathcal{MU}_C}=0$, $y_{\mathcal{MU}_C}=0$, $z_{\mathcal{MU}_C}=0$\\
\For{$i \in \{1, 2, \cdots, n \}$}{Find $\operatorname{mu}(I^{(i)})$ (see Algorithm 1 in \cite{Panagiotou2015})\\

Find the length, breadth and height of $\operatorname{mu}(I^{(i)})$ and set,\\
$x=x_{\operatorname{mu}(I^{(i)})}, \quad y=y_{\operatorname{mu}(I^{(i)})}, \quad z=z_{\operatorname{mu}(I^{(i)})}$.\\ 

Reset values as,\\
$x_{\mathcal{MU}_C}=\operatorname{max}(x_{\mathcal{MU}_C},x),\quad y_{\mathcal{MU}_C}=\operatorname{max}(y_{\mathcal{MU}_C},y),\quad z_{\mathcal{MU}_C}=max(z_{\mathcal{MU}_C},z).$
}
Construct $\mathcal{MU}_C$ with dimensions $x_{\mathcal{MU}_C}\times y_{\mathcal{MU}_C}\times z_{\mathcal{MU}_C}$\\
Assign a list $\mathcal{L}_C =[ \quad ]$\\
\For{$i \in \{1, 2, \cdots, n \}$}{Append $\mathcal{L}_C$ by all images of the $i^{th}$ chain that intersect $\mathcal{MU}_C$}
\Return{$\mathcal{L}_C$}
\caption{Construction of $\mathcal{L}_C$ (minimal periodic link)}
\label{AlgoH}
\end{algorithm}
\begin{figure}[ht!]
    \centering
     \includegraphics[scale=0.03]{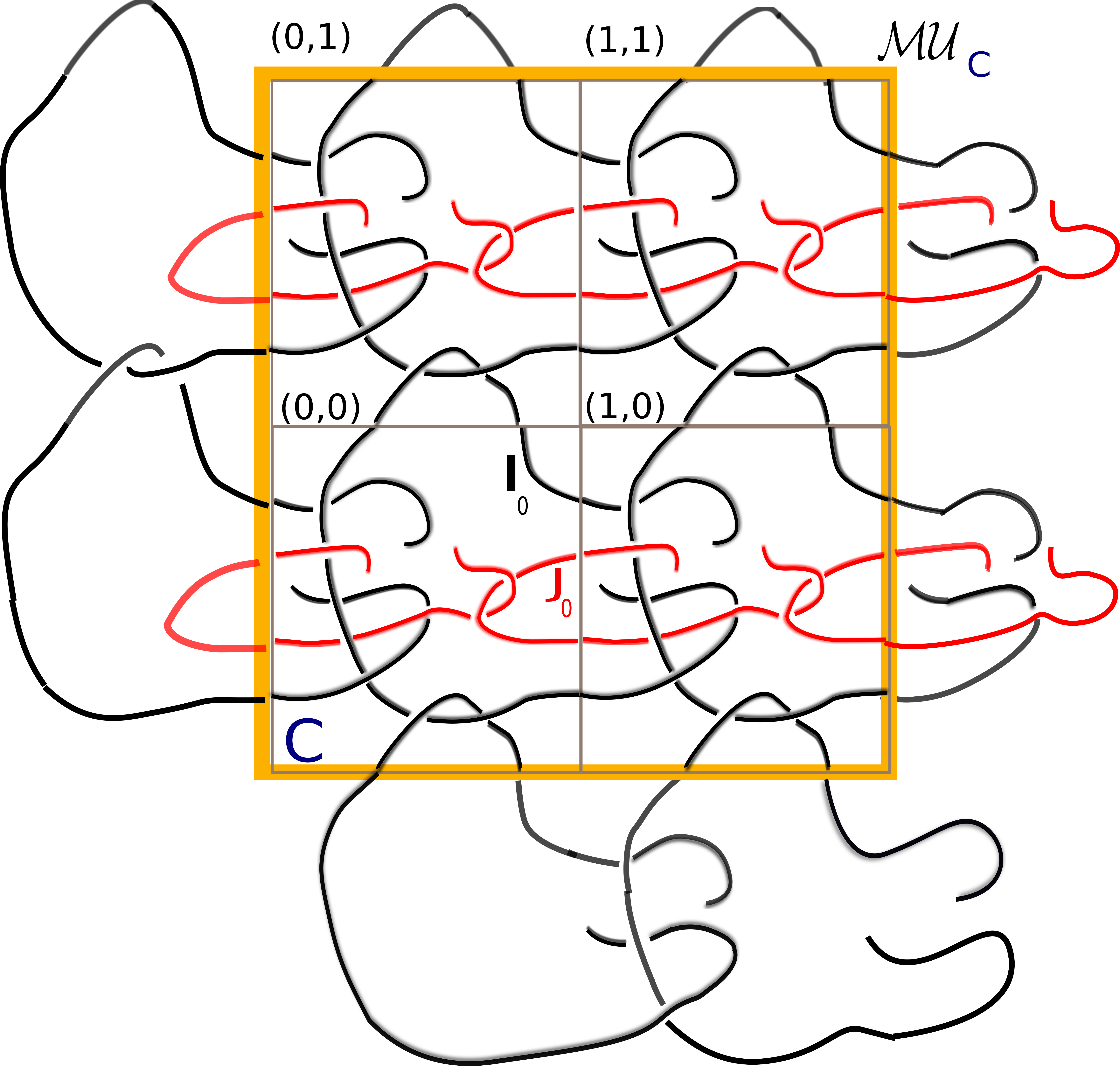}
    \caption{Let $(0,0)$ denote the base cell, $C$, of the 2 PBC system in Figure(\ref{PBCterms}). $\operatorname{mu}(I_0)$ comprises of the cells labelled by $(0,1), (0,0)$ and $(1,0)$ and $\operatorname{mu}(J_0)$ comprises of the cells labelled by $(0,0)$ and $(1,0)$. The region bounded by the cells $(0,1), (0,0), (1,1)$ and $(1,0)$ forms the minimal collective unfolding, $\mathcal{MU}_C$ of the PBC system. The link comprising of $I_0, J_0$ and all other images of the chains $I$ and $J$ that intersect $\mathcal{MU}_C$ is the minimal periodic link, $\mathcal{H}_0$.}
    \label{mu1}
\end{figure}

\noindent The minimal periodic link , $\mathcal{L}_C$, in a PBC system has the following properties :
\begin{enumerate}
    \item It contains at least one image of each free chain.
    \item It contains the images of all free chains  that may impose topological constraints to an image of any free chain.
    \item It is independent of the position of the minimal collective unfolding in the periodic system with respect to the frame of reference.
\end{enumerate}  

\begin{corollary}
    Given a unit cell, $C$, of a periodic system, the minimal periodic link, $\mathcal{L}_C$, is the minimal link (with respect to the number of components) that satisfies the above properties.
\end{corollary}
\begin{proof}
      Let $\mathcal{L}$ be an arbitrary link extracted from the infinite system which satisfies the above properties.  Let $I$ be a free chain in the periodic system. Then, by properties i and ii,  $\mathcal{L}$ contains an image, say $I_0$, of $I$, such that all the chains in the periodic system that impose topological constraints to $I_0$ are also contained in  $\mathcal{L}$. Therefore, all the chains that intersect $\operatorname{mu}(I_0)$ belong to $\mathcal{L}$ for any free chain $I$. Moreover, to ensure that all the chains that interact topologically with $I_0$ are accounted for, the chains that intersect a convex set that contains $\operatorname{mu}(I_0)$ must be in $\mathcal{L}$ (Consider for example a closed chain $I_0$ whose Seifert surface intersects $\operatorname{mu}(I_0)$).    Therefore,  $\mathcal{L}$ must be the collection of chains that intersect a convex union of cells that contains all minimal unfoldings. By Proposition \ref{mu_smallest}, $\mathcal{MU}_C$ is the smallest such cell. This implies $\mathcal{L}=\mathcal{L}_C$.
\end{proof}

The definition of the minimal periodic link enables us to assign a Jones polynomial to a periodic system, which satisfies the main requirements mentioned in the introduction of this section.

\begin{definition}(Periodic Jones polynomial) Consider a periodic system generated by a cell $C$. The Periodic Jones polynomial, $\mathsf{V}_P(C)$, is defined as the Jones polynomial of the minimal periodic link, $\mathcal{L}_C$, namely, $\displaystyle  \mathsf{V}_P(C) := \mathsf{V}(\mathcal{L}_C).$
\label{PJPinf}
\end{definition}
\noindent The Periodic Jones polynomial, defined as above, is a (finite) one variable polynomial and has the following properties:
\begin{enumerate}
\item It captures the topological complexity of any image of any free chain.
    \item It captures the total topological complexity imposed on any image of a free chain. 
    \item It is independent of the position of the minimal collective unfolding with respect to the frame of reference.
    \item It is a polynomial with integer coefficients, for systems of closed curves, and it is a topological invariant  almost everywhere (see Remark \ref{ae}).
    \item It is a polynomial with real coefficients, for systems of open curves, and is a continuous function of the curve coordinates almost everywhere (see Remark \ref{ae}).
    \item When the images of the free chains do not touch the boundary of the generating cell, the Periodic Jones polynomial equals the Jones polynomial of the link of curves contained completely within the cell. In this case, the Cell Jones polynomial and the Periodic Jones polynomial coincide.
\end{enumerate}

\begin{remark}
\label{inf_chains}
As mentioned in Remark \ref{infbase}, in the case of systems of infinite free chains, the image of a free chain (and its minimal unfolding) depend on the choice of base point. Thus, the Periodic Jones polynomial of a periodic system of infinite free chains becomes well defined only when a base point for each generating chain in the system is specified. For a periodic system of infinite free chains with a choice of base point for each parent image, the Periodic Jones polynomial has the same properties as that of finite open free chains in PBC. One may choose the base points to maximize the number of components in the minimal periodic link for the Periodic Jones polynomial of a system.
\end{remark}

\begin{remark}\label{ae}
The minimal periodic link and the Periodic Jones polynomial of a PBC system are uniquely determined by its base cell, $C$ (and a choice of base points in the case of infinite chains). Under deformation of a given system, the relative positions of the chains to the cell, or the cell itself, may change and this could change the minimal periodic link (which is defined by the minimal unfoldings of chains relative to a cell). That may result in a change of the number of components in the minimal periodic link, which would lead to a discontinuity in the Periodic Jones polynomial. However, such discontinuities occur for a set of measure zero in the space of configurations (corresponding to configurations when a vertex of a polygonal chain intersects one of the faces of a cell to enter a new cell, changing its minimal unfolding). A normalization of the polynomial on the number of components can be used in such situations (see Section \ref{modulo}) to partially account for the change of the number of components in the unit cell. Another approach to address deformations in periodic systems in time without discontinuities, is to keep the components of the minimal periodic link fixed to that of time $t=0$ and study its evolution and the evolution of its Periodic Jones polynomial in time, without re-evaluating the minimal periodic link relative to the cell in time. Then, the Periodic Jones polynomial is an invariant for closed chains and a continuous function of the chain coordinates for open and infinite chains without any discontinuities or changes of number of components.
\end{remark}

\begin{corollary} The Periodic Jones Polynomial of a PBC system defined by a cell $C$, can be expressed as,
$\quad \displaystyle \mathsf{V}_P(C) :=  \int_{\vec{\xi}\in S^2} \mathsf{V}_P(C_{\vec{\xi}})dS =  \int_{\vec{\xi}\in S^2} \mathsf{V}\left((\mathcal{L}_C)_{\vec{\xi}}\right)dS$, \quad
where $\displaystyle \mathsf{V}_{P}(C_{\vec{\xi}}):=\mathsf{V}\left((\mathcal{L}_C)_{\vec{\xi}}\right)$ is the diagrammatic Periodic Jones polynomial with respect to $\vec{\xi}\in S^2$.
\end{corollary}
\begin{proof}
    It follows from  Definitions \ref{jones-open-def} and \ref{PJPinf}.
\end{proof}

\subsection{Normalization of the Jones polynomial}
\label{modulo}
In this section we discuss a normalization of the Jones polynomial as a means to
compare, in terms of entanglement complexity, systems with different number of components.

\begin{definition}(Normalized Jones polynomial)
Let  $\mathcal{L}_n$ be a link (linkoid) with $n$ components. The Normalized Jones polynomial of $\mathcal{L}_n$ is defined as :
\begin{equation}
    \mathsf{NV}(\mathcal{L}_n) := \mathsf{V}(\mathcal{L}_n) d^{-(n-1)}
    \label{norm_JP}
\end{equation}
where the numerator, $\mathsf{V}(\mathcal{L}_n)$, is the Jones polynomial of $\mathcal{L}_n$ and the denominator, $d^{n-1}$, is the Jones polynomial of the trivial link with $n$ components.
\label{normalized_JP}
\end{definition}

In Equation \ref{norm_JP}, the fraction $\mathsf{NV}(\mathcal{L}_n) \in \mathbb{R}(A^2)$ is not necessarily a polynomial. By implementing Euclidean division of polynomials, it is possible to express $\mathsf{V}(\mathcal{L}_n)$ as $\mathsf{V}(\mathcal{L}_n) = d^{n-1} q(A^2) + r(A^2)$,
where $\displaystyle q(A^2), r(A^2) \in \mathbb{R}[A^2]$ are the quotient and the remainder, respectively. Dividing both sides by $d^{n-1}$, we get $\mathsf{NV}(\mathcal{L}_n) =  q(A^2) + r(A^2)d^{-(n-1)}$

Thus, we infer the following :
\begin{enumerate}
   \item If $r(A^2) = 0$, then the link, $\mathcal{L}_n$, is a disjoint union of $n$ copies of a knot whose Jones polynomial is equal to $\displaystyle \sqrt[n]{q(A^2)}$. If $r(A^2) \neq 0$, then the $n$ components in the link are not all disjoint (see Example \ref{ex_qr_hopf}).
   \item If $\mathcal{L}_n$ is a collection of closed curves in 3-space, the Normalized Jones polynomial is a topological invariant. If $\mathcal{L}_n$ is a collection of open curves in 3-space, the Normalized Jones polynomial is a continuous function of the curve coordinates.
\end{enumerate}

Since the remainder of the normalized Jones (and Periodic Jones) polynomial is zero only for links with a Jones polynomial equal to that of the trivial link of any number of components, we will use it as a proxy of a measure of comparison between the complexity of the Jones polynomials of links with different number of components. This is demonstrated in the following example, where normalization is applied to the Jones polynomial of the Hopf link:

\begin{example}
The Hopf link consists of 2 components and depending on their relative orientations, the Jones polynomial is either $-A^2-A^{10}$ or $-A^{-2}-A^{-10}$. Whereas, the Jones polynomial of the trivial link with 2 components is, $d=-A^{2}-A^{-2}$. We can rewrite the Jones polynomial of the Hopf link as follows: $-A^2-A^{10} = d\times (A^8-A^4+2) +(2A^{-2})$ and $-A^{-2}-A^{-10} = d\times 0 + (-A^{-2}-A^{-10})$.
The remainder terms for both the expressions of the Jones polynomial of the Hopf link are non-zero, while the remainder of the unlink is zero, implying that the 2 components of the Hopf link (with any orientation) are interlinked. 
\label{ex_qr_hopf}
\end{example}

The normalization procedure can be applied to the Periodic Jones polynomial (as well as to the Cell Jones polynomial) in order to compare among systems with different number of components in the corresponding minimal periodic links.

\begin{definition}(Normalized Periodic Jones polynomial) For a PBC system defined by a cell $C$, whose minimal periodic link is $\mathcal{L}_C$, the Normalized Periodic Jones polynomial, $\mathsf{NV}_P(C)$ is defined as:
\begin{equation}
    \mathsf{NV}_P(C) := \mathsf{V}_P(C) d^{-(|\mathcal{L}_C|-1)}
    \label{NPJP}
\end{equation}
where, $\mathsf{V}_P(C)=\mathsf{V}(\mathcal{L}_C)$ and  $|\mathcal{L}_C|$ is the number of components in $\mathcal{L}_C$.
\label{NPJPinf}
\end{definition}

\section{Application of the Periodic Jones polynomial and the Cell Jones polynomial to systems employing PBC}\label{sec_app}

In this section we show how the Periodic Jones polynomial and the Cell Jones polynomial can be applied in practice to measure multi-chain entanglement of open and closed chains in systems employing PBC. In Section \ref{ex_inf_doubly}, we apply the Cell Jones polynomial and the Periodic Jones polynomial to examples of 3D realizations of textile patterns (doubly-periodic structures). In Section \ref{polymer_sys}, we apply the Periodic Jones polynomial to examples of polymer melt systems of linear chains of varying molecular weights obtained through a molecular dynamics simulation. Throughout this section, we present the Periodic Jones polynomial and the Cell Jones polynomial in terms of the variable $A$ without substituting $A=t^{-1/4}$. Different polynomials, either in the coefficients, the powers of $A$, or both, indicate different configurations. The span of a polynomial is the difference between the maximum exponent of $A$ and the minimum exponent of $A$. We will use the span as a measure of topological complexity of a system, assuming that higher span is indicative of higher topological complexity \cite{knot-book}. 

\subsection{Doubly-periodic structures}
\label{ex_inf_doubly}
Doubly-periodic structures are complex 3 dimensional entangled networks made of curves, which are infinite, intertwined threads \cite{Fukuda2022,Fukuda2023,Morton2009, Markande2020, Knittel2020, Wadekar2020, Wadekar2021}.  Let us consider two kinds of such structures (weaves), \textit{single jersey} and \textit{twill}, as shown  with cartoon representations in Figures \ref{tex1a} and \ref{tex1b}, respectively. We generate two specific 3D realizations of the systems, we denote their cells $\mathcal{C}_{jersey}$ and $\mathcal{C}_{twill}$, respectively, whose generating arcs are determined by the coordinates $l_{C_{jersey}}=\{(0,2,0.2),(1,2,0.2),(1,3,-0.2)\}$,
$\{(1,0,-0.2),(0.5,1,0),$\\ $(0.5,2.5,0), (1.5,2.8,0.1), (2.5,2.5,0), (2.5,0.5,0),(2,0,-0.2)\}$, 
$\{(2,3,-0.2),(2,2,0.2),(3,2,0.2)\}$ and\\
$l_{C_{twill}}=\{(3,4,-0.2),(2,3,0),(2,1,0),(4,1,0.2)\}$, $\{(0,1,0.2),(1,1,0.2),(1,2,0.2),(3,2.5,0.1),(3.1,1.8,0),$\\$(3,0,-0.2)\}$, 
respectively. We consider the base point of each generating chain to be the first point of its first arc in the list of coordinates. The Cell Jones polynomials of the two systems are:
\begin{equation}
    \begin{split}\displaystyle
     \mathsf{V}_{\mathcal{C}_{jersey}}&=0.02A^{-14}+0.08A^{-12}+0.02A^{-10}+0.08A^{-8}+0.26A^{-6}+0.34A^{-4}\\&+0.64A^{-2}+1.5+0.62A^{2}+0.26A^{4}+0.16A^{6}+0.02A^{8}.
    \end{split}
    \label{cell_1a}
\end{equation}
\begin{equation}
    \begin{split}\displaystyle
    \mathsf{V}_{\mathcal{C}_{twill}}&=-0.04A^{-2}+0.08-0.92A^{2}-0.1A^{4}-0.08A^{6}-0.08A^{8}-0.86A^{10}.\\
    \end{split}
    \label{cell_1b}
\end{equation}
Equations \ref{cell_1a} and \ref{cell_1b} show that the jersey and the twill systems are topologically distinct. For the given choice of base cell and parent image for each system, the results indicate that $span(\mathsf{V}_{\mathcal{C}_{jersey}})=22$ and $span(\mathsf{V}_{\mathcal{C}_{twill}})=12$. This indicates higher complexity in the jersey system compared to the twill. 
However, this may be reflecting the fact that the the cell of the jersey system consists of 3 arcs, while that of the twill consists of 2. In order to decrease the reflection of the number of components in the polynomials, we evaluate the Normalized Cell Jones polynomial of the jersey and twill systems, to better compare between the two systems (see Equations \ref{ncjp_j}, \ref{ncjp_t}).
\begin{equation}
\begin{split}\displaystyle
    \mathsf{NV}_{\mathcal{C}_{jersey}} &= 0.22+0.16A^{2}+0.02A^{4} \quad + \quad
    {d}^{-2} (0.02A^{-14}+0.08A^{-12}\\&+0.02A^{-10} +0.08A^{-8}+0.26A^{-6}+0.12A^{-4} +0.48A^{-2}+1.04+0.3A^{2}).
\end{split}    
\label{ncjp_j}
\end{equation}
\begin{equation}
\begin{split}\displaystyle
    \mathsf{NV}_{\mathcal{C}_{twill}} &= 1.7+0.02A^{2}-0.78A^{4}+0.08A^{6}+0.86A^{8} \quad + {d}^{-1} (1.66A^{-2}+0.1).
\end{split}        
\label{ncjp_t}
\end{equation}
In the above equations, the remainder polynomials are $r(\mathsf{NV}_{\mathcal{C}_{jersey}})=0.02A^{-14}+0.08A^{-12}+0.02A^{-10}+0.08A^{-8}+0.26A^{-6}+0.12A^{-4} +0.48A^{-2}+1.04+0.3A^{2}$ and $r(\mathsf{NV}_{\mathcal{C}_{twill}})=1.66A^{-2}+0.1$, and their spans are given as, $span(r(\mathsf{NV}_{\mathcal{C}_{jersey}}))=16$ and $span(r(\mathsf{NV}_{\mathcal{C}_{twill}}))=2$, respectively. Thus, the normalization for the number of components reinforces that the jersey is more complex than the twill.

To capture the global complexity of the systems, we examine the Periodic Jones polynomial. The minimal periodic links  of the two systems are shown in Figure \ref{wv_1_mpl}. The Periodic Jones polynomial for the jersey and the twill systems are given in Equations \ref{mpl_1a} and \ref{mpl_1b}, respectively.  Thus, for the given choice of base cell and base points for each system, $span(\mathsf{V}_P(\mathcal{C}_{jersey}))=86$ and $span(\mathsf{V}_P(\mathcal{C}_{twill}))=26$. This is indicative of higher global complexity in the jersey system as compared to the twill. 

\begin{figure}[ht!]
    \centering
     \raisebox{3 cm}{\includegraphics[height=2cm,width=2cm]{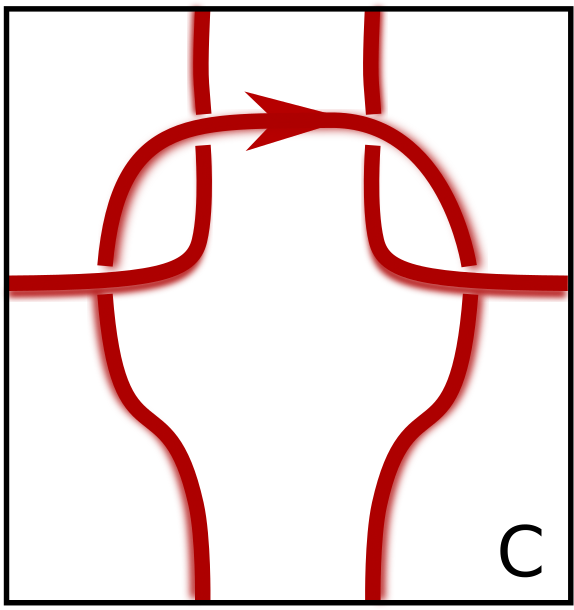}} \hspace{1cm}
    \includegraphics[height=8cm,width=9cm]{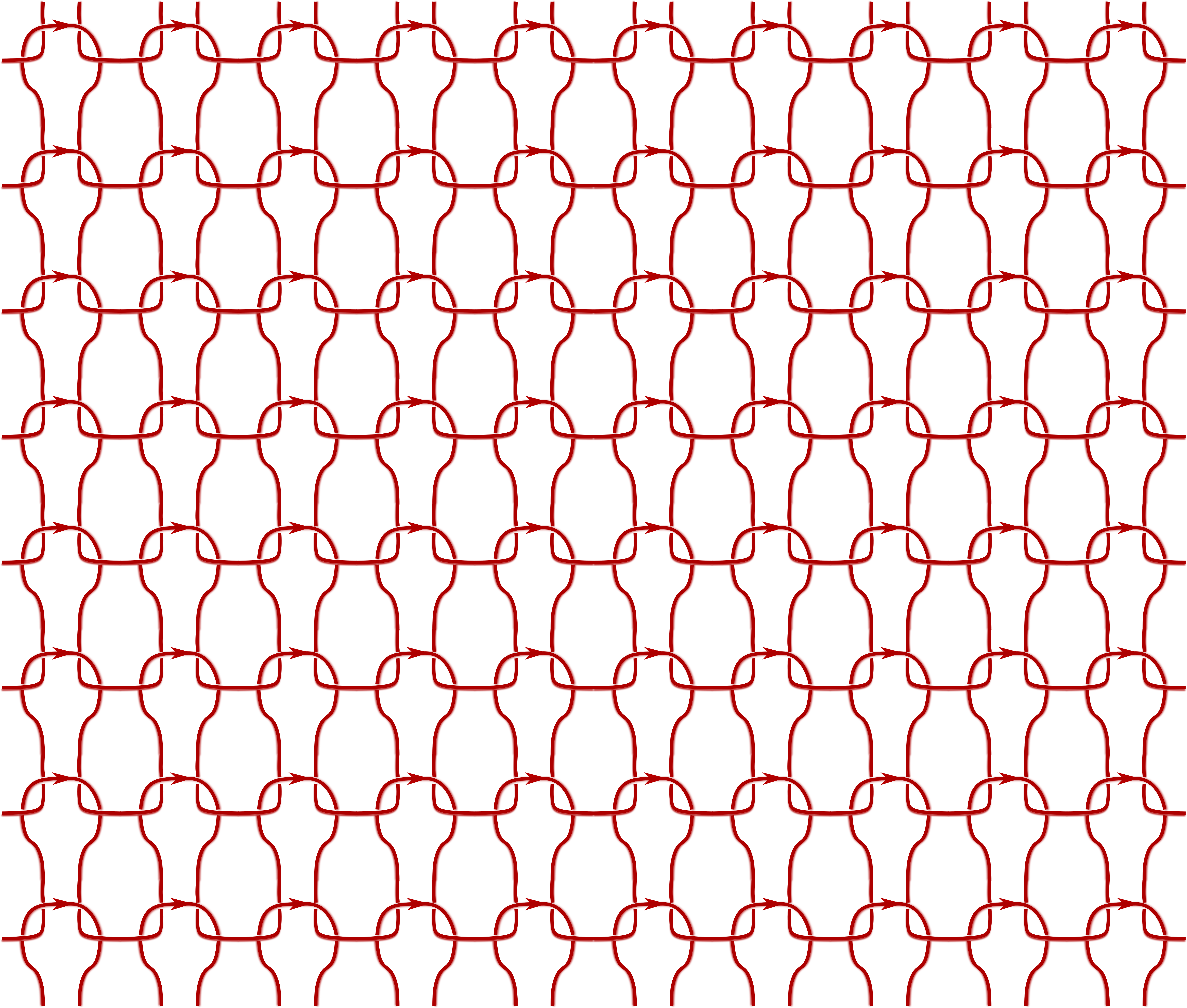}
    \caption{Left: A cartoon representative of the generating cell of a doubly-periodic structure (single jersey). Right: A cartoon representation of the doubly-periodic system extended in space.}
    \label{tex1a}
\end{figure}

\begin{equation}
    \begin{split}\displaystyle     \mathsf{V}_P(\mathcal{C}_{jersey})&=-0.029A^{-68}-0.029A^{-66}-0.118A^{-64}-0.029A^{-62}-0.088A^{-60}-0.059A^{-58}\\&-0.265A^{-56}-0.324A^{-54}+0.059A^{-52}+0.176A^{-50}+0.206A^{-48}-0.235A^{-46}\\&-0.118A^{-44}-1.971A^{-42}-1.176A^{-40}-2.324A^{-38}-5.265A^{-36}+4.529A^{-34}\\&+15.353A^{-32}-11.765A^{-30}-40.412A^{-28}-9.382A^{-26}+39.353A^{-24}+6.147A^{-22}\\&-55.853A^{-20}-39.265A^{-18}+22.412A^{-16}+23.265A^{-14}-23.176A^{-12}-30.706A^{-10}\\&-8.294A^{-8}+3.794A^{-6}-0.471A^{-4}-3.824A^{-2}-2.412-2.824A^{2}-2.353A^{4}\\&-1.853A^{6}-0.176A^{8}+0.853A^{10}+0.882A^{12}+0.029A^{14}-0.176A^{16}-0.088A^{18}.
    \end{split}
    \label{mpl_1a}
\end{equation}

\begin{figure}[ht!]
    \centering
    \raisebox{3 cm}{\includegraphics[height=2cm,width=2cm]{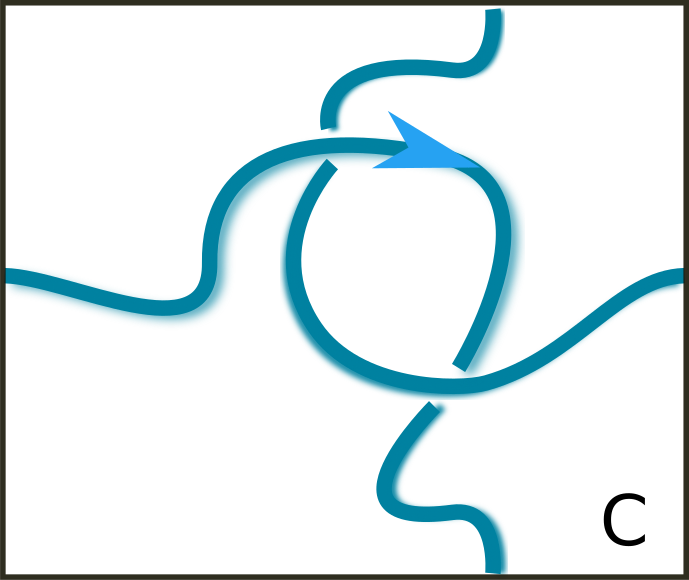}} \hspace{1cm}
    \includegraphics[height=9cm,width=9cm]{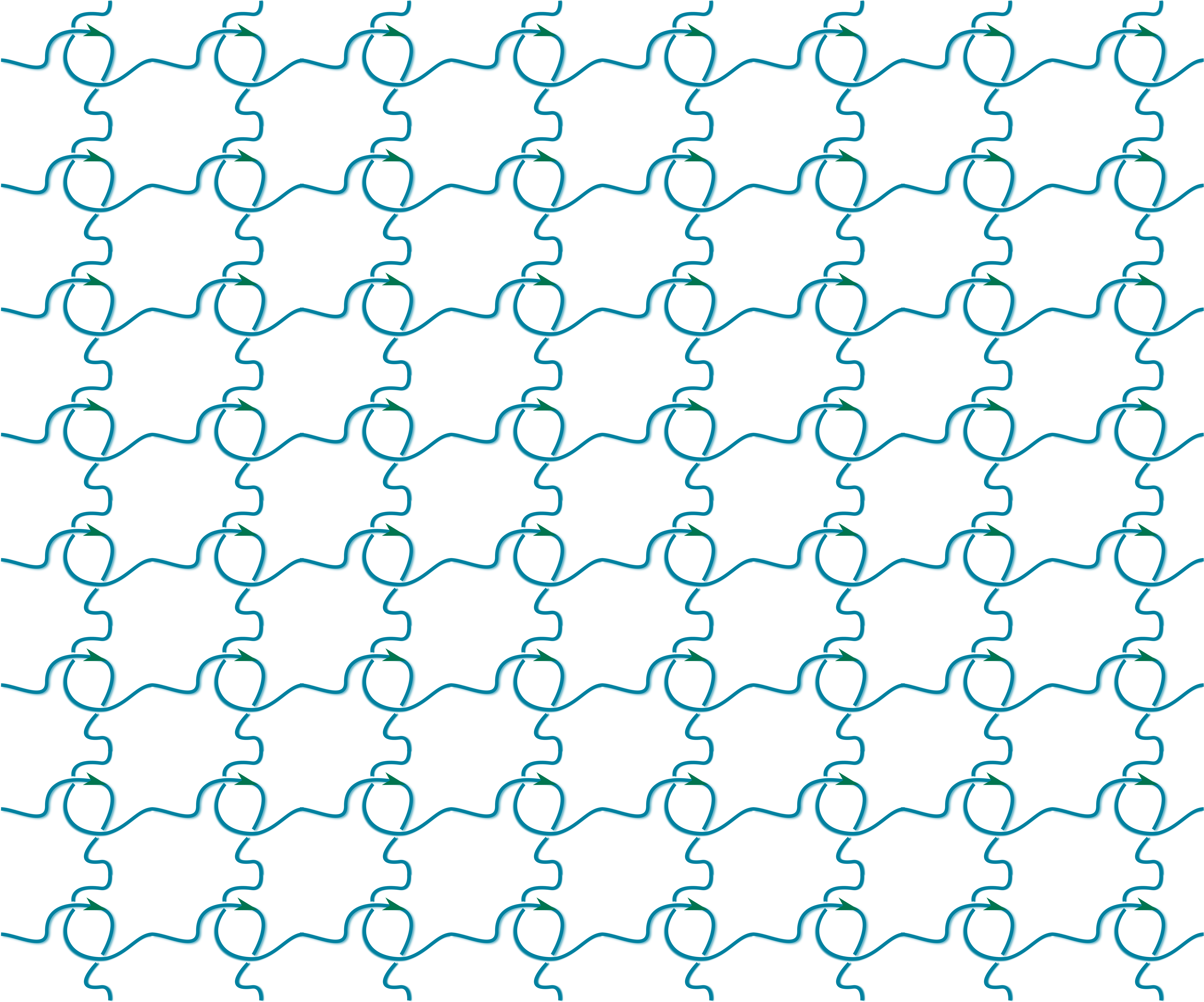}
       \caption{Left: A cartoon representative of the generating cell of a doubly-periodic structure of (twill). Right: A cartoon representation of the doubly-periodic system extended in space.}
    \label{tex1b}
\end{figure}
\begin{equation}
    \begin{split}\displaystyle    \mathsf{V}_P(\mathcal{C}_{twill})&=0.02A^{-6}+0.02A^{-4}+0.12A^{-2}+0.88+0.76A^{2}+0.88A^{4}\\&+0.96A^{6}+0.1A^{8}+0.14A^{12}+0.1A^{14}+0.02A^{20}.
    \end{split}
    \label{mpl_1b}
\end{equation}
\begin{figure}[ht!]
    \centering
    \includegraphics[scale=0.07]{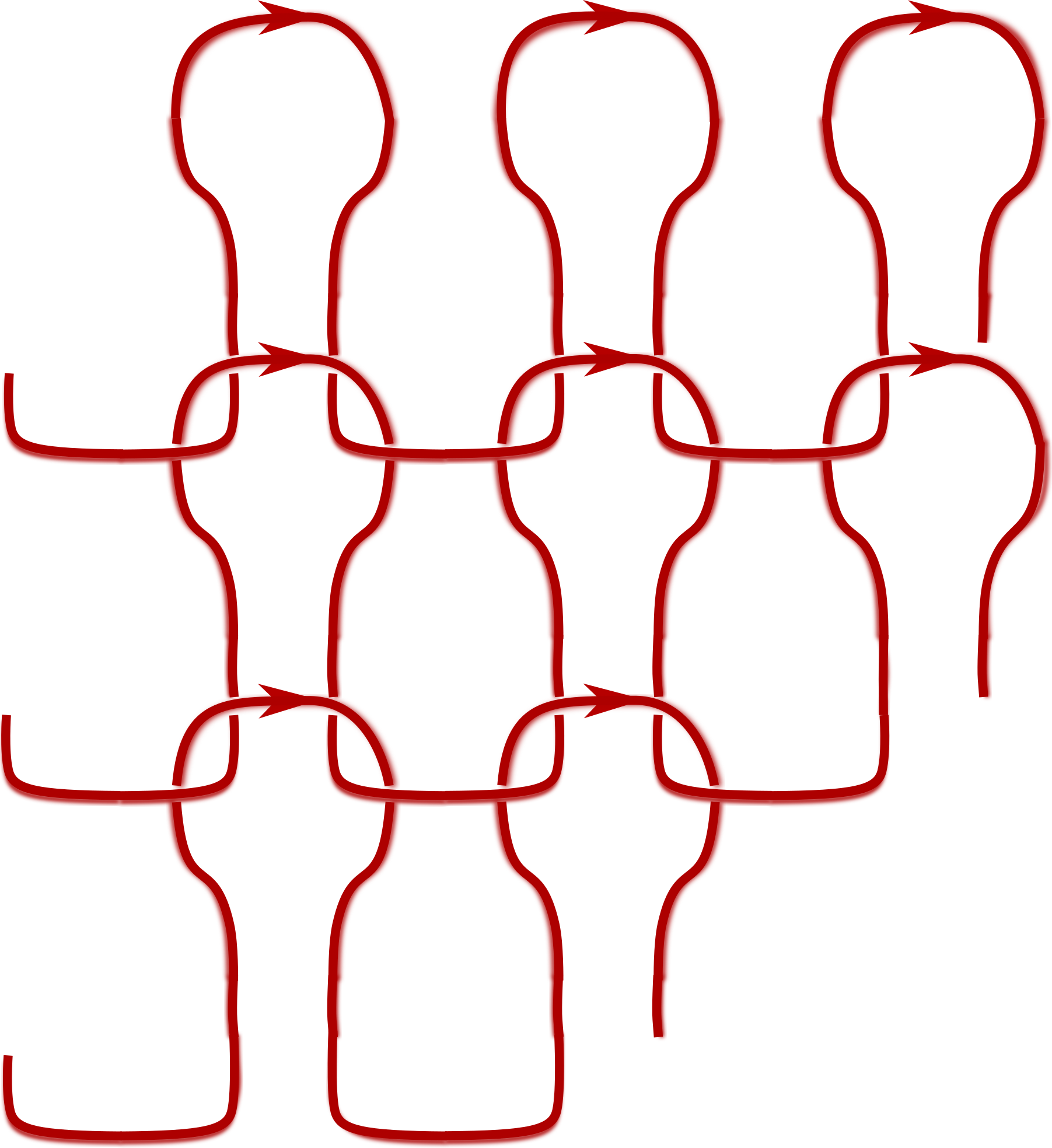} \quad \quad
    \includegraphics[scale=0.07]{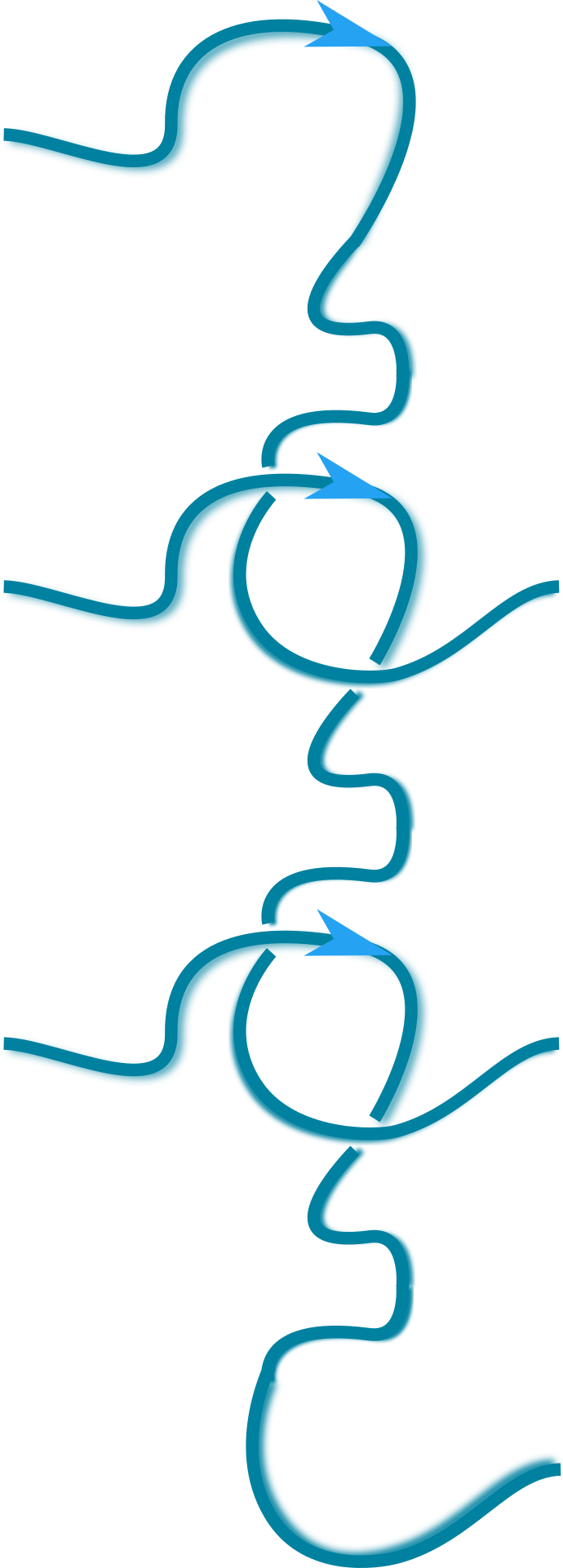}
    \caption{Left : The minimal periodic link containing 8 components of the system in Figure \ref{tex1a} (jersey). Right : The minimal periodic link containing 3 components of the system in Figure \ref{tex1b}(twill).}
    \label{wv_1_mpl}
\end{figure}

The minimal periodic link for the jersey has 8 components, while that of the twill has 3 components. In order to take that into account, the Normalized Periodic Jones polynomials of the jersey and twill systems are computed, in Equations \ref{npjp_j} and \ref{npjp_t}, respectively.

\begin{equation}\displaystyle
\begin{split}
    \mathsf{NV}_P(\mathcal{C}_{jersey}) &= -0.647+0.176A^{2}+0.088A^{4}\\
    & +
    {d}^{-7} (-0.029A^{-68}-0.029A^{-66}-0.118A^{-64}-0.029A^{-62}\\&-0.088A^{-60}-0.059A^{-58}-0.265A^{-56}-0.324A^{-54}+0.059A^{-52}+0.176A^{-50}\\&+0.206A^{-48}-0.235A^{-46}-0.118A^{-44}-1.971A^{-42}-1.176A^{-40}-2.324A^{-38}\\&-5.265A^{-36}+4.529A^{-34}+15.353A^{-32}-11.765A^{-30}-40.412A^{-28}-9.382A^{-26}\\&+39.353A^{-24}+6.147A^{-22}-55.853A^{-20}-39.265A^{-18}+22.412A^{-16}\\&+22.618A^{-14}-23.0A^{-12}-35.147A^{-10}-7.059A^{-8}-9.176A^{-6}\\&+3.235A^{-4}-24.618A^{-2}+3.765-22.382A^{2}+3.824A^{4}-12.353A^{6}\\&+3.529A^{8}-1.824A^{10}+2.118A^{12}).
\end{split}    
\label{npjp_j}
\end{equation}
\begin{equation}\displaystyle
\begin{split}
    \mathsf{NV}_P(\mathcal{C}_{twill}) &= 1.2+1.26A^{2}-0.26A^{4}-0.2A^{6}+0.2A^{8}+0.1A^{10}-0.04A^{12}+0.02A^{16}\\
    & + \quad {d}^{-2} (0.02A^{-6}-1.18A^{-4}-1.14A^{-2}-1.26-1.56A^{2}).
\end{split}        
    \label{npjp_t}
\end{equation}
In the above equations, the span of the remainder polynomials of the Normalized Periodic Jones polynomial of the jersey and the twill systems are $span(r(\mathsf{NV}_P(\mathcal{C}_{jersey})))=80$ and $span(r(\mathsf{NV}_P(\mathcal{C}_{twill})))=8$, respectively. This further supports that the global complexity of the jersey system is higher than that of the twill.
\subsection{Polymer melts of linear chains}
\label{polymer_sys}
In this section, we analyze the multi-chain entanglement of linear FENE polymer chains in a melt in equilibrium, obtained through molecular dynamics simulations employing 3 PBC at temperature $T=1$ and dimensionless density $\rho=0.84$ under a Lennard Jones potential and employing the velocity verlet algorithm. The systems analyzed consist in 100 generating chains in PBC each. The Cell Jones polynomial of such a system would correspond to the Jones polynomial of a link with at least $100$ components and the Periodic Jones polynomial has even more, a computationally expensive computation that goes beyond the scope of this section. We focus instead on a subset of chains in the melt that do not intersect the periodic boundary, for which the Cell Jones polynomial is equal to the Periodic Jones polynomial and the number of components is low. More precisely, we consider subsystems within those polymer melt samples, with $7$ generating chains, such that the chains unfold within the cell. Let us denote the cells $\mathbf{P}_{10}$, $\mathbf{P}_{20}$ and $\mathbf{P}_{30}$ (see Figures \ref{poly10}, \ref{poly20} and \ref{poly30}) such that the molecular weight per chain in the system is $10$, $20$ and $30$, respectively.  Notice that the number of components in these systems are the same, thus there is no reason to examine their normalized polynomials. The Periodic Jones polynomials for these systems and are given in Equations \ref{10cell}, \ref{20cell} and \ref{30cell}.

\begin{figure}[ht!]
    \centering
     \raisebox{0.5 em}{\includegraphics[height=2.2cm,width=2.2cm]{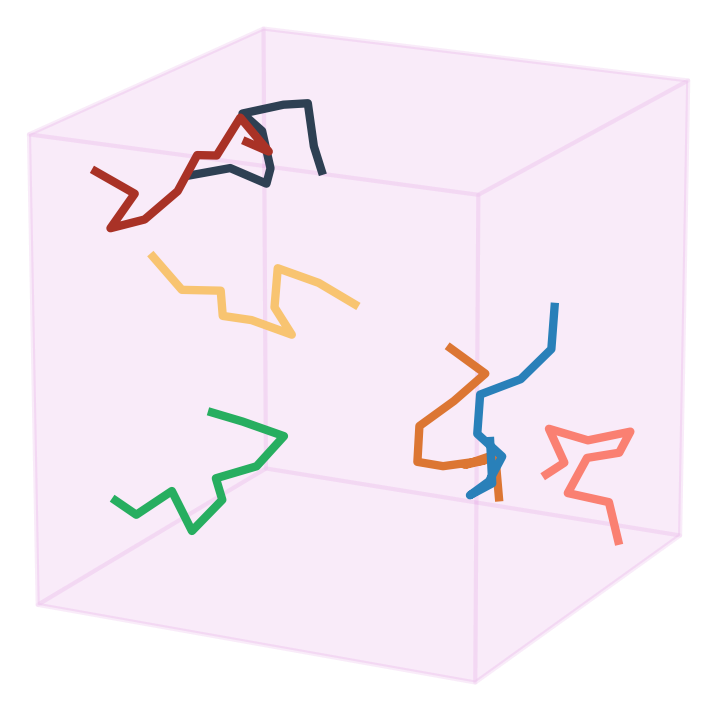}} \quad \includegraphics[scale=0.25]{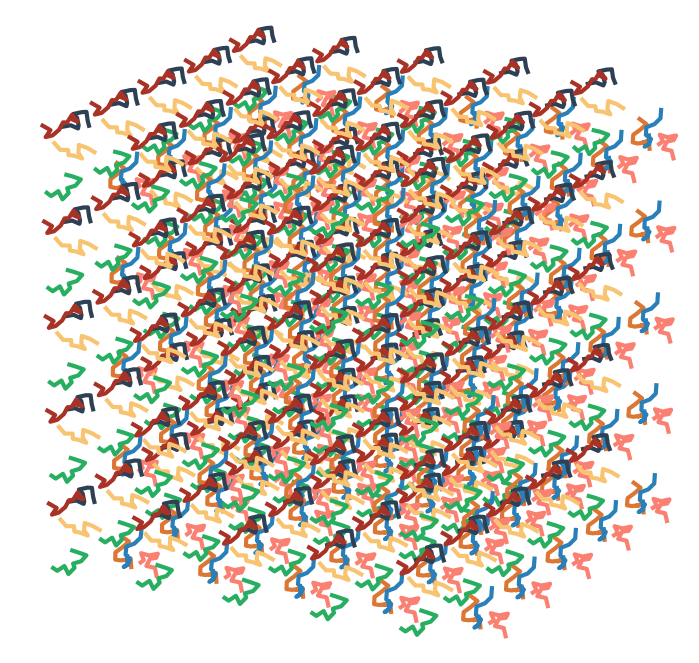}
    \caption{Left: Selected chains in the simulation unit cell of a polymer melt of molecular weight 10 (system $\mathbf{P}_{10}$); Right : The 3 PBC system $\mathbf{P}_{10}$ extended in space.  }
    \label{poly10}
\end{figure}

\begin{figure}[ht!]
    \centering
     \raisebox{0.6 em}{\includegraphics[height=2.5cm,width=2.5cm]{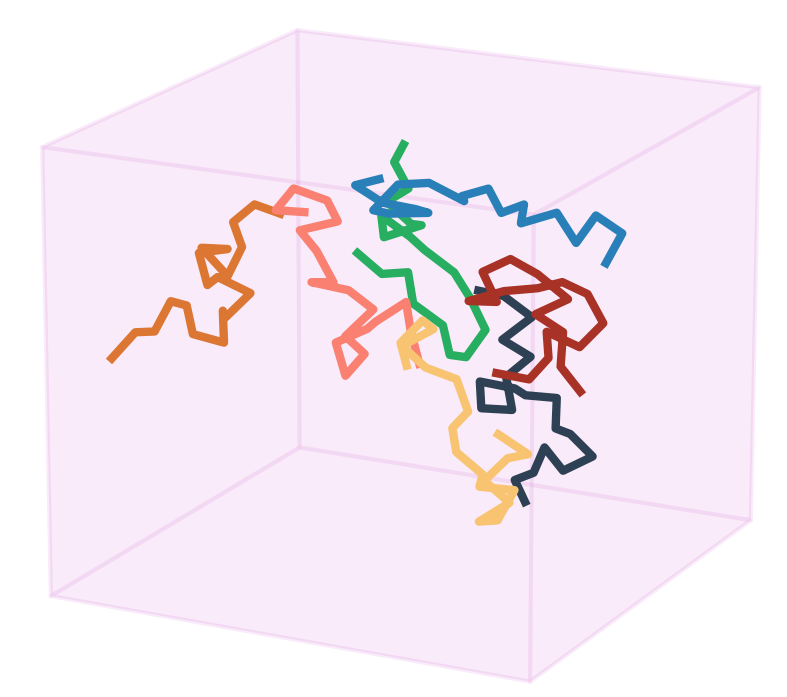}} \quad \includegraphics[scale=0.3]{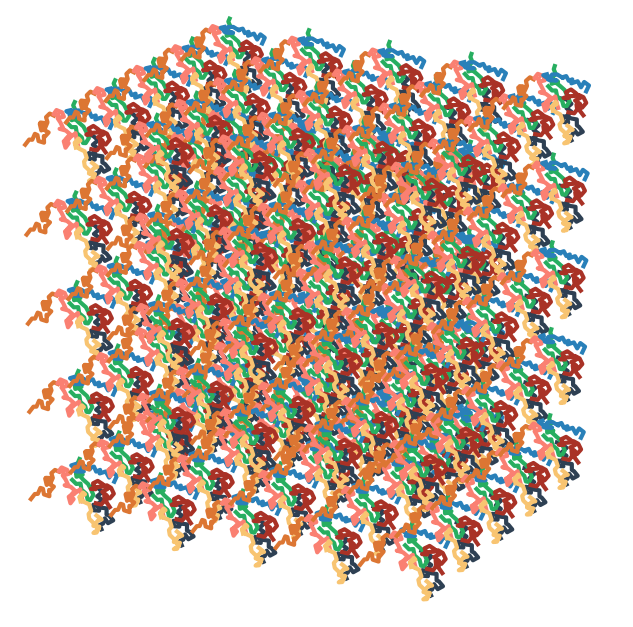}
    \caption{Left: Selected chains in he simulation unit cell of a polymer melt of molecular weight 20  (system $\mathbf{P}_{20}$); Right : The 3 PBC system $\mathbf{P}_{20}$ extended in space.  }
    \label{poly20}
\end{figure}

\begin{figure}[ht!]
    \centering
     \raisebox{0.6 em}{\includegraphics[height=2.5cm,width=2.5cm]{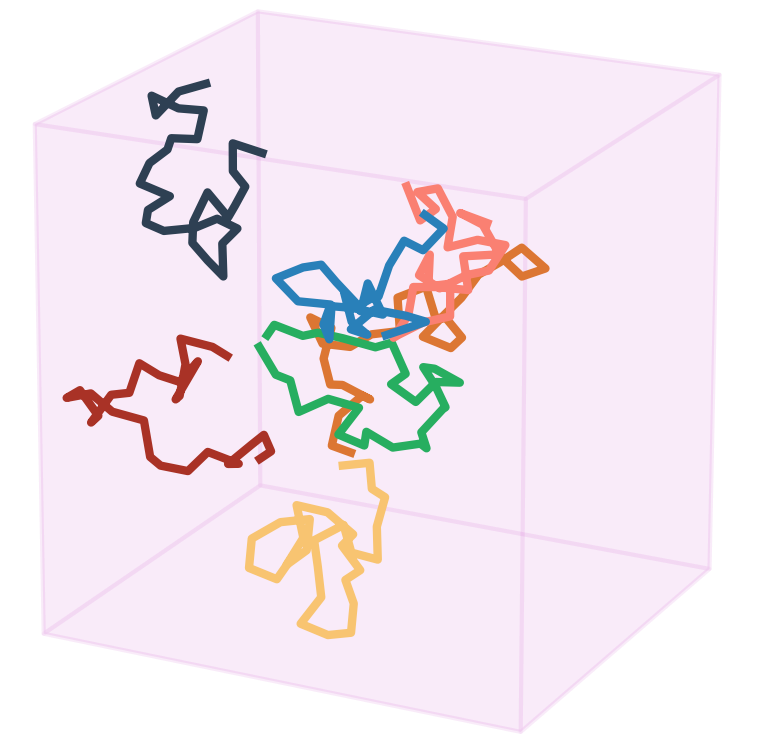}} \quad \includegraphics[scale=0.25]{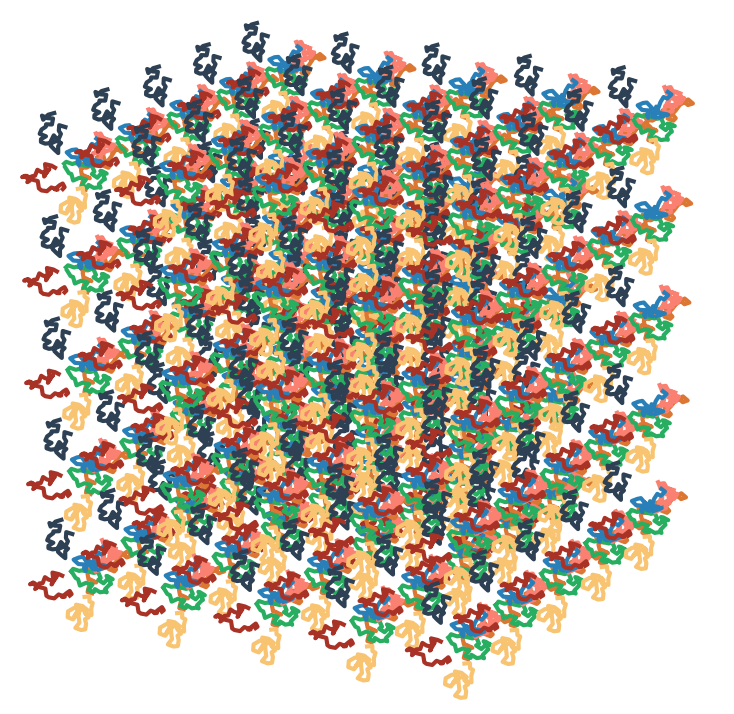}
    \caption{Left: Selected chains in the simulation unit cell of a polymer melt of molecular weight 30 (system $\mathbf{P}_{30}$); Right : The 3 PBC system $\mathbf{P}_{30}$ extended in space.  }
    \label{poly30}
\end{figure}

\begin{equation}
   \begin{split}\displaystyle
    \mathsf{V}_P(\mathbf{P}_{10})&=-0.06A^{-22}-0.34A^{-18}+0.04A^{-16}-0.48A^{-14}+0.88A^{-12}+0.66A^{-10}\\&+4.52A^{-8}+3.08A^{-6}+10.72A^{-4}+4.98A^{-2}+13.92+5.12A^{2}\\&+10.32A^{4}+3.66A^{6}+4.2A^{8}+1.62A^{10}+0.8A^{12}+0.32A^{14}+0.04A^{16}.
    \end{split}
    \label{10cell}
\end{equation}

\begin{equation}
   \begin{split}\displaystyle
    \mathsf{V}_P(\mathbf{P}_{20})&=-0.068A^{-22}-0.364A^{-18}+0.182A^{-16}-0.227A^{-14}+1.614A^{-12}+1.773A^{-10}\\&+5.841A^{-8}+4.545A^{-6}+11.295A^{-4}+5.023A^{-2}+12.818+3.5A^{2}\\&+8.886A^{4}+2.227A^{6}+3.886A^{8}+1.295A^{10}+1.114A^{12}+0.432A^{14}\\&+0.182A^{16}+0.045A^{18}.
    \end{split}
    \label{20cell}
\end{equation}

\begin{equation}
   \begin{split}\displaystyle
    \mathsf{V}_P(\mathbf{P}_{30})&=-0.03A^{-24}-0.03A^{-22}-0.121A^{-20}+0.091A^{-16}+0.545A^{-14}\\&+1.364A^{-12}+1.97A^{-10}+4.515A^{-8}+3.424A^{-6}+9.303A^{-4}+4.242A^{-2}\\&+12.424+4.788A^{2}+10.394A^{4}+4.03A^{6}+5.061A^{8}+1.697A^{10}\\&+1.061A^{12}-0.212A^{16}-0.242A^{18}-0.182A^{20}-0.061A^{22}-0.03A^{24}.
    \end{split}
    \label{30cell}
\end{equation}

\noindent  Since the above polynomials are distinct from one another, we infer that the 3 systems are different. Moreover, $\displaystyle span\left(\mathsf{V}_P(\mathbf{P}_{10})\right) = 38 < span\left(\mathsf{V}_P(\mathbf{P}_{20})\right) = 40 < span\left(\mathsf{V}_P(\mathbf{P}_{30})\right) = 48$,
which suggests an increase in multi-chain topological entanglement with increasing molecular weight. 

\section{The Periodic Jones polynomial at the limit of the infinite periodic system}\label{sec_mpl_inf}

In this section, we discuss to what extent the Periodic Jones polynomial, which is the Jones polynomial of a finite link, captures the global topological complexity present in the infinite periodic system. Notice that since the Jones polynomial of the infinite system diverges, it is meaningful to instead examine the Jones polynomial of a cutoff or arbitrary size of the periodic system. We will show that for a system with 1 closed chain in 1 PBC, generated by a cell $C$, the Periodic Jones polynomial in combination with the periodic linking number fully determines the contribution of one of the states of the Jones polynomial of any (appropriate) cutoff of the infinite periodic system.  More precisely:

\begin{theorem} (Jones polynomial for any cutoff of a system with 1 closed chain in 1 PBC)
Let $C$ be the base cell of a periodic system generated by 1 closed chain in 1 PBC and let $\mathcal{MU}_C$ be its minimal collective unfolding consisting of $|\mathcal{MU}_C|$ copies of $C$ and $\mathcal{L}_C$ its minimal periodic link. Let $\mathcal{N}_{\mathcal{MU}_C}$ be $(2N-1)( |\mathcal{MU}_C|) - (N-1)$ repeated copies of the base cell $C$ (along the PBC faces) and let $(\mathcal{L}_{C})_N$, denote the link composed by all images of $I$ intersecting or contained in $\mathcal{N}_{\mathcal{MU}_C}$.  The Jones polynomial of $(\mathcal{L}_C)_{N}$ can be expressed as
\begin{equation}
    \begin{split}\displaystyle
        \mathsf{V}\left((\mathcal{L}_C)_{N}\right)  &=(-A)^{-(N-1)\mathsf{SLK}_P(\mathcal{L}_C)} d^{N-1} {\mathsf{V}_P(C)}^N + \tilde{\Lambda}\\
    \end{split}
    \label{VHN-exp}
\end{equation}
\noindent where ${\mathsf{V}_P(C)}$ is the Periodic Jones polynomial of $C$, $\mathsf{SLK}_P(\mathcal{L}_C)$ denotes the periodic linking with self images of copies of $\mathcal{L}_C$. The term $(-A)^{-(N-1)\mathsf{SLK}_P(\mathcal{L}_C)} d^{N-1} {\mathsf{V}_P(C)}^N$ is the contribution of one of the states of the Jones polynomial of $(\mathcal{L}_C)_{N}$ which leads to disconnection of all copies of $\mathcal{L}_{C}$ in $(\mathcal{L}_{C})_N$ and the remainder term, $\tilde{\Lambda}$, is contributed by the states which do not lead to disconnection of the $N$ 
 copies of $\mathcal{L}_{C}$ in $(\mathcal{L}_{C})_N$.
\label{corr1PBCjones}
\end{theorem}
\begin{proof}
Notice that in $\mathcal{N}_{\mathcal{MU}_C}$, which consists of $(2N-1)( |\mathcal{MU}_C|) - (N-1)$ cells, we can identify exactly $N$ minimal unfoldings (consisting of $|\mathcal{MU}_C|$ cells each) separated by $|\mathcal{MU}_C|-1$ cells (Thus we have $(N-1)(|\mathcal{MU}_C|-1)$ intermediate cells). Each minimal unfolding defines a minimal periodic link, $\displaystyle \mathcal{L}_C$, and, due to the separation between these minimal unfoldings, these minimal periodic links have no common components. Indeed, a component that intersects $\mathcal{MU}_C$ can intersect at most $|\mathcal{MU}_C|-1$ cells not in $\mathcal{MU}_C$. Moreover, the components that belong in these minimal periodic links are exactly all the components intersecting or inside $\mathcal{N}_{\mathcal{MU}_C}$, since any component in the $|\mathcal{MU}_C|-1$ intermediate cells must intersect one of the neighboring $\mathcal{MU}_C$.

Let $\vec{\xi}\in S^2$ and project the periodic system to the plane with normal vector $\vec{\xi}$. Let $c$ denote the crossings in ${(\mathcal{L}_C)}_{\vec{\xi}}$, we call them self-crossings, and let $k$ be the crossings between any consecutive pair of disjoint (in terms of components) copies of $\displaystyle {(\mathcal{L}_C)}_{\vec{\xi}}$, we call them shared crossings.  Then
 $\displaystyle \left \langle {((\mathcal{L}_C)_{N})}_{\vec{\xi}} \right \rangle$ can be expanded as, $\displaystyle \sum_{\mathcal{T}} \langle 
\mathcal{T} \rangle$, where $\mathcal{T}$ is any state obtained by resolving the $(N-1)k$ shared crossings in $\displaystyle {((\mathcal{L}_{C})_{N})}_{\vec{\xi}}$. Any such $\mathcal{T}$ is in fact a knot/link with a total of $\displaystyle Nc$ unresolved crossings.  Among the $2^{(N-1)k}$ such states, there exists a unique state, $\mathcal{S}$, which leads to the disconnection of  the $N$ copies of ${(\mathcal{L}_C)}_{\vec{\xi}}$ in  $\displaystyle {((\mathcal{L}_C)_{N})}_{\vec{\xi}}$. Smoothing along the shared crossings between any two copies of ${(\mathcal{L}_C)}_{\vec{\xi}}$ by respecting orientation, contributes a factor $A$, if the crossing is positive or a factor $A^{-1}$, if the crossing is negative. Thus, disconnecting a copy of  ${(\mathcal{L}_C)}_{\vec{\xi}}$ accumulates a factor $A^\theta$, where $\theta= 2\mathsf{SLK}_P({(\mathcal{L}_C)}_{\vec{\xi}})= 2\sum_{\vec{v}} \mathsf{L}({(\mathcal{L}_C)}_{\vec{\xi}}, {(\mathcal{L}_C)}_{\vec{\xi}} + \vec{v})$ (where $\mathsf{L}$ denotes the half algerbaic sum of shared crossings) the . Therefore, $A^{(N-1)2\mathsf{SLK}_P({(\mathcal{L}_C)}_{\vec{\xi}})}$ is the net smoothing factor associated with the state, $\mathcal{S}$. Then,
\begin{equation}
    \begin{split}\displaystyle
        \left\langle {((\mathcal{L}_C)_{N})}_{\vec{\xi}} \right\rangle  &= A^{2(N-1)\mathsf{SLK}_P({(\mathcal{L}_C)}_{\vec{\xi}})} d^{N-1} \left\langle {(\mathcal{L}_C)}_{\vec{\xi}} \right\rangle^N + \Lambda,
    \end{split}
    \label{eq-theta-lam}
\end{equation}

\noindent where $\Lambda$ is the contribution to the Jones polynomial from the remaining $2^{(N-1)k}-1$ states.

The Jones polynomial is obtained upon the normalisation of Equation \ref{eq-theta-lam} by the diagrammatic writhe of $ {((\mathcal{L}_C)_{(N)})}_{\vec{\xi}}$ :
\begin{equation}
\begin{split}\displaystyle \mathsf{Wr}\left({((\mathcal{L}_C)_{(N)})}_{\vec{\xi}}\right) &= \sum_{i=1}^{N} \mathsf{Wr}\left((\mathcal{L}_C^i)_{\vec{\xi}}\right) + \sum_{i,j=1, i\neq j}^{N-1} \mathsf{L}\left((\mathcal{L}_C^i)_{\vec{\xi}},(\mathcal{L}_C^{j})_{\vec{\xi}}\right)\\&= N \mathsf{Wr}\left({(\mathcal{L}_C)}_{\vec{\xi}}\right) + (N-1) \mathsf{SLK}_P\left({(\mathcal{L}_C)}_{\vec{\xi}}\right),
\end{split}
\label{wrr1}
\end{equation}
where $\mathsf{Wr}\left({(\mathcal{L}_C^i)}_{\vec{\xi}}\right)$ is the diagrammatic writhe of a copy of ${(\mathcal{L}_C)}_{\vec{\xi}}$; $\mathsf{L}\left((\mathcal{L}_C^i)_{\vec{\xi}},(\mathcal{L}_C^{j})_{\vec{\xi}}\right)$ is the diagrammatic linking number between the $i^{th}$ and $j^{th}$ copies of ${(\mathcal{L}_C)}_{\vec{\xi}}$ and $\mathsf{SLK}_P\left({(\mathcal{L}_C)}_{\vec{\xi}}\right)$ is the periodic linking  with self images of ${(\mathcal{L}_C)}_{\vec{\xi}}$ (accounting only shared crossings). The expression for the Jones polynomial of $(\mathcal{L}_C)_{N}$ is then 

\begin{equation}
    \begin{split}\displaystyle
        \mathsf{V}\left((\mathcal{L}_C)_{N}\right) &= {(-A^3)}^{-\mathsf{Wr}\left({((\mathcal{L}_C)_{N})}_{\vec{\xi}}\right)} \left\langle {((\mathcal{L}_C)_{N})}_{\vec{\xi}} \right\rangle\\
         &=(-A)^{-(N-1)\mathsf{SLK}_P\left({(\mathcal{L}_C)}_{\vec{\xi}}\right)} d^{N-1} {\mathsf{V}\left({(\mathcal{L}_C)}_{\vec{\xi}}\right)}^N + \tilde{\Lambda}\\
        &=(-A)^{-(N-1)\mathsf{SLK}_P\left({(\mathcal{L}_C)}\right)} d^{N-1} {\mathsf{V}\left({(\mathcal{L}_C)}\right)}^N + \tilde{\Lambda},\\
    \end{split}
\end{equation}
where $\displaystyle \tilde{\Lambda} = \Lambda {(-A^3)}^{-\mathsf{Wr}\left({((\mathcal{L}_C)_{N})}_{\vec{\xi}}\right)}$. 
\end{proof}

\begin{remark}\label{rem-23pbc-more}
Theorem \ref{corr1PBCjones} can be extended to systems involving more closed generating chains and/or more PBC, with some modifications. The Periodic Jones polynomial and the Periodic Linking Number again contribute to the Jones polynomial of one state of the $N$th cutoff, but with an extra term that is a polynomial of a smaller link than the minimal periodic link. More precisely, for a system of closed chains with 2 PBC (resp. 3 PBC), $\mathcal{N}_{\mathcal{MU}_C}$ can be defined so as to contain $N^2$ (resp. $N^3|\mathcal{MU}_C|$) minimal collective unfoldings separated by smaller collections of intermediate cells. In this case, the chains that intersect $\mathcal{N}_{\mathcal{MU}_C}$ are $N^2$ (resp. $N^3$) copies of $\mathcal{L}_C$, but also some chains that completely unfold in the intermediate cells. By resolving the shared crossings between every $\mathcal{L}_C$ and the rest of the chains, a state $\mathcal{S}$ can be obtained which, is a disjoint union of $N^2$ (resp. $N^3$) copies of $\mathcal{L}_C$ and a link consisting of the images of chain(s) which are not part of any copy of $\mathcal{L}_C$. This remainder link is smaller than the minimal periodic link and thus can be expected to contribute a simpler polynomial factor in the Jones polynomial of that state.
\end{remark}

\begin{remark}
Theorem \ref{corr1PBCjones} does not hold for systems involving open or infinite curves in PBC, unless if we make a convention discussed below. This is because, unlike systems generated by closed chains, the result depends on all the projection directions used for the computation of the polynomial. For every two minimal periodic links, there is always a projection direction where they overlap and, even if one is completely over the other, their Jones polynomial cannot be expressed as that of the product of two disjoint components. The theorem would hold if we allowed this property for linkoids.
\end{remark}

\section{Conclusions}\label{sec_conc}
In this manuscript, a generalization of the Jones polynomial, the Periodic Jones polynomial, is introduced to measure topological entanglement in collections of physical filaments with periodicity. The Periodic Jones polynomial measures the topological complexity of a finite link in the system which is minimal in the number of components, but captures the total entanglement imposed on an image of each chain of the system. By being defined on the periodic system in 3-space, rather than an identification space, the Periodic Jones polynomial is well defined for both open and closed curves and it is a continuous function of the curve coordinates or a topological invariant, respectively, almost everywhere. For the special case of closed chains in 1 PBC, it is proved that the Periodic Jones polynomial is a repetitive factor, up to a remainder, in the Jones polynomial of any (appropriately chosen) cutoff of arbitrary size of the infinite system. More precisely, it is proved that the Periodic Jones polynomial and the Periodic Linking number \cite{Panagiotou2015} fully determine the contribution of one of the states of the Jones polynomial of any cutoff of the periodic system. 

The Cell Jones polynomial is also introduced and discussed as a method to measure collective local entanglement in a periodic system that is computationally less expensive and can capture the entanglement present within a unit cell.

Two examples of periodic systems and their Periodic Jones polynomial and the Cell Jones polynomial are presented. Namely, the Periodic Jones polynomial is computed for 3D realizations of doubly-periodic textile patterns and systems of linear polymers in a melt obtained from molecular dynamic simulations employing PBC. Our results show that the Periodic Jones polynomial can be used to compare textile patterns and classify them according to topological complexity. In the context of polymers, it is shown that the Periodic Jones polynomial captures increasing collective periodic entanglement in polymers with increasing molecular weight. All these results point to how the Periodic Jones polynomial is a useful topological parameter of an infinite periodic system of entangled filaments.
 
\section{Acknowledgements}
Kasturi Barkataki and Eleni Panagiotou were supported by NSF (Grant No. DMS-1913180 and NSF CAREER 2047587).

\bibliographystyle{plain}
\bibliography{paperDatabase}

\begin{thebibliography}{10}

\bibitem{knot-book}
C.~C. Adams.
\newblock {\em The knot book: {A}n elementary introduction to the mathematical theory of knots}.
\newblock American MAthematical Society, 2004.

\bibitem{Arsuaga2005}
J.~Arsuaga, M.~Vazquez, P.~McGuirk, S.~Trigueros, D.~W. Sumners, and J.~Roca.
\newblock {DNA} knots reveal a chiral organization of {DNA} in phage capsids.
\newblock {\em Proc. Natl. Acad. Sci. (USA)}, 102:9165--9169, 2005.

\bibitem{Barkataki2022}
Kasturi Barkataki and Eleni Panagiotou.
\newblock The {Jones} polynomial of collections of open curves in 3-space.
\newblock {\em Proceedings of the Royal Society A}, 478(2267):20220302, 2022.

\bibitem{Castle2011}
T.~Castle, M.~E. Evans, and S.~T. Hyde.
\newblock Entanglement of embedded graphs.
\newblock {\em Progress of Theor. Physics (Supplement)}, 191:235--244, 2011.

\bibitem{Delgado2017}
O.~Delgado-Friedrichs, S.~Hyde, M.~O'Keeffe, and O.~Yaghi.
\newblock Crystal structures as periodic graphs: the topological genome and graph databases.
\newblock {\em Structural Chemistry}, 28:39--44, 2017.

\bibitem{Edwards1967}
S.~F. Edwards.
\newblock Statistical mechanics with topological constraints: I.
\newblock {\em Proc. Phys. Soc.}, 91:513--9, 1967.

\bibitem{Evans2013b}
M.~E. Evans, V.~Robins, and S.~T. Hyde.
\newblock Periodic entanglement {II}: weavings from hyperbolic line patterns.
\newblock {\em Acta Chryst.}, A69:262--275, 2013.

\bibitem{Evans2015a}
Myfanwy~E. Evans and Stephen~T. Hyde.
\newblock {Periodic entanglement III: tangled degree-3 finite and layer net intergrowths from rare forests}.
\newblock {\em Acta Crystallographica Section A}, 71(6):599--611, Nov 2015.

\bibitem{Evans2013a}
Myfanwy~E. Evans, Vanessa Robins, and Stephen~T. Hyde.
\newblock {Periodic entanglement I: networks from hyperbolic reticulations}.
\newblock {\em Acta Crystallographica Section A}, 69(3):241--261, May 2013.

\bibitem{Evans2015b}
Robins~Vanessa Evans Myfanwy~E. and Hyde~Stephen T.
\newblock Ideal geometry of periodic entanglements.
\newblock {\em Proc. R. Soc. A.}, 471, 2015.

\bibitem{Freyd1985}
P.~Freyd, D.~Yetter, J.~Hoste, W.~Lickorish, K.~C. Millett, and A.~Ocneanu.
\newblock A new polynomial invariant for knots and links.
\newblock {\em Bull. Am. Math. Soc.}, 12:239--46, 1985.

\bibitem{Fukuda2022}
Mizuki Fukuda, Motoko Kotani, and Sonia Mahmoudi.
\newblock Construction of weaving and polycatenane motifs from periodic tilings of the plane.
\newblock {\em arXiv preprint arXiv:2206.12168}, 2022.

\bibitem{Fukuda2023}
Mizuki Fukuda, Motoko Kotani, and Sonia Mahmoudi.
\newblock Classification of doubly periodic untwisted (p,q)-weaves by their crossing number and matrices.
\newblock {\em J. Knot Theory Ramif.}, 32(5):1, 2023.

\bibitem{Gauss1877}
K.~F. Gauss.
\newblock {\em Werke}.
\newblock Kgl. Gesellsch. Wiss. G\"ottingen, 1877.

\bibitem{Igram2016}
S.~Igram, K.~C. Millett, and E.~Panagiotou.
\newblock Resolving critical degrees of entanglement in olympic rings systems.
\newblock {\em J. Knot Theory Ramif.}, 25:14, 2016.

\bibitem{Kauffman1990}
L.~H. Kauffman.
\newblock An invariant of regular isotopy.
\newblock {\em Trans. Amer. Math. Soc.}, 318:417--471, 1990.

\bibitem{Kauffman2001}
L.~H. Kauffman.
\newblock {\em Knots and Physics}, volume~1 of {\em Series on knots and everything}.
\newblock World Scientific, 1991.

\bibitem{Knittel2020}
C.~Knittel, M.~Tanis, A.~L. Stoltzfus, T.~Castle, R.~D. Kamien, and G.~Dion.
\newblock Modeling textile structures using bicontinuous surfaces.
\newblock {\em J. Math. and the Arts}, 14:1--14, 2020.

\bibitem{Kolbe2022}
B.~Kolbe and M.~Evans.
\newblock Enumerating isotopy classes of tilings guided by the symmetry of triply periodic minimal surfaces.
\newblock {\em SIAM J. Appl. Alg. and Geom.}, 6(1):1--40, 2022.

\bibitem{Li2019}
Z.~Li, L.~Liu, M.~Zheng, J.~Zhao, N.~C. Seeman, and C.~Mao.
\newblock Making engineered {3D DNA} crystals robust.
\newblock {\em J. Am. Chem. Soc.}, 141:15850–15855, 2019.

\bibitem{Liu2018}
Y.~Liu, M.~O'Keeffe, M.~Treacy, and O.~Yaghi.
\newblock The geometry of periodic knots, polycatenanes and weaving from a chemical perspective: a library for reticular chemistry.
\newblock {\em Chem. Soc. Rev.}, 47:4642--4664, 2018.

\bibitem{Markande2020}
Shashank Markande and Elisabetta~A. Matsumoto.
\newblock Knotty knits are tangles on tori.
\newblock {\em arXiv: Soft Condensed Matter}, 2020.

\bibitem{Millett2016}
K.~C. Millett and E.~Panagiotou.
\newblock Entanglement transitions in one dimensional confined flows.
\newblock {\em Fluid Dyn. Res.}, 50:011416, 2016.

\bibitem{Morton2009}
H.~R. Morton and S.~Grishanov.
\newblock Doubly periodic textile structures.
\newblock {\em J. Knot Theory Ramif.}, 18:1597--1622, 2009.

\bibitem{OKeeffe2021}
Michael O'Keeffe and M.~M.~J. Treacy.
\newblock Isogonal piecewise linear embeddings of 1-periodic weaves and some related structures.
\newblock {\em Acta Crystall. Section A}, 77(2):130--137, 2021.

\bibitem{OKeeffe2022}
Michael O'Keeffe and M.~M.~J. Treacy.
\newblock Isogonal piecewise-linear embeddings of 1-periodic knots and links, and related 2-periodic chain-link and knitting patters.
\newblock {\em Acta Crystall. Section A}, 78(3):234--241, 2022.

\bibitem{Panagiotou2015}
E.~Panagiotou.
\newblock The linking number in systems with periodic boundary conditions.
\newblock {\em J. Comput. Phys.}, 300:533--573, 2015.

\bibitem{Panagiotou2020b}
E.~Panagiotou and L.~Kauffman.
\newblock Knot polynomials of open and closed curves.
\newblock {\em Proc. R. Soc. A}, 476:20200124, 2020.

\bibitem{Panagiotou2021}
E.~Panagiotou and L.~Kauffman.
\newblock Vassiliev measures of complexity for open and closed curves in 3-space.
\newblock {\em Proc. R. Soc. A}, 477:20210440, 2021.

\bibitem{Panagiotou2014}
E.~Panagiotou and M.~Kr\"oger.
\newblock Pulling-force-induced elongation and alignment effects on entanglement and knotting characteristics of linear polymers in a melt.
\newblock {\em Phys. Rev. E}, 90:042602, 2014.

\bibitem{Panagiotou2019}
E.~Panagiotou, K.~C. Millett, and P.~J. Atzberger.
\newblock Topological methods for polymeric materials: characterizing the relationship between polymer entanglement and viscoelasticity.
\newblock {\em Polymers}, 11:11030437, 2019.

\bibitem{Panagiotou2013}
E.~Panagiotou, K.~C. Millett, and S.~Lambropoulou.
\newblock Quantifying entanglement for collections of chains in models with periodic boundary conditions.
\newblock {\em Procedia IUTAM: Topological Fluid Dynamics}, 7:251--260, 2013.

\bibitem{Panagiotou2020}
E.~Panagiotou and K.~W. Plaxco.
\newblock A topological study of protein folding kinetics.
\newblock {\em Topology of Biopolymers, AMS Contemporary Mathematics Series}, page 223, 2020.

\bibitem{Panagiotou2011}
E.~Panagiotou, C.~Tzoumanekas, S.~Lambropoulou, K.~C. Millett, and D.~N. Theodorou.
\newblock A study of the entanglement in systems with periodic boundary conditions.
\newblock {\em Progr. Theor. Phys. Suppl.}, 191:172--181, 2011.

\bibitem{P2018}
Eleni {Panagiotou} and Kenneth~C. {Millett}.
\newblock {Linking matrices in systems with periodic boundary conditions}.
\newblock {\em Journal of Physics A Mathematical General}, 51(22):225001, June 2018.

\bibitem{Przytycki1987}
J.~Przytycki and P.~Traczyk.
\newblock Conway algebras and skein equivalence of links.
\newblock {\em Proc. Amer. Math. Soc.}, 100:744--48, 1987.

\bibitem{Qin2011}
J.~Qin and S.~T. Milner.
\newblock Counting polymer knots to find the entanglement length.
\newblock {\em Soft Matter}, 7:10676--93, 2011.

\bibitem{rosi2005rod}
Nathaniel~L Rosi, Jaheon Kim, Mohamed Eddaoudi, Banglin Chen, Michael O'Keeffe, and Omar~M Yaghi.
\newblock Rod packings and metal- organic frameworks constructed from rod-shaped secondary building units.
\newblock {\em Journal of the American Chemical Society}, 127(5):1504--1518, 2005.

\bibitem{Sulkowska2012}
J.~I. Sulkowska, E.~J. Rawdon, K.~C. Millett, J.~N. Onuchic, and A.~Stasiak.
\newblock Conservation of complex knotting and slipknotting in patterns in proteins.
\newblock {\em Proc. Natl. Acad. Sci.}, 109:E1715, 2012.

\bibitem{Taylor1974}
J.~B. Taylor.
\newblock Relaxation of toroidal plasma and generation of reverse magnetic fields.
\newblock {\em Phys. Rev. Lett.}, 33:1139, 1974.

\bibitem{Wadekar2021}
P.~Wadekar, C.~Amanatides, L.~Kapillani, G.~Dion, R.~D. Kamien, and D.~E. Breen.
\newblock Geometric modeling of complex knitting stitches using a bicontinuous surface and its offsets.
\newblock {\em Comp. Aided Geom. Des.}, 89:102024, 2021.

\bibitem{Wadekar2020}
P.~Wadekar, P.~Goel, C.~Amanatides, G.~Dion, R.~D. Kamien, and D.~E. Breen.
\newblock Geometric modeling of knitted fabrics using helicoid scaffolds.
\newblock {\em J. Eng. Fivers and Fabrics}, 15:1558925020913871, 2020.

\bibitem{Wang2022}
J.~Wang and E.~Panagiotou.
\newblock The protein folding rate and the topology and geometry of the native state.
\newblock {\em Scientific Reports}, 12:6384, 2022.

\bibitem{Zhang2022}
C.~Zhang, M.~Zheng, Y.~P. Ohayon, S.~Vecchioni, R.~Sha, N.~C. Seeman, N.~Jonoska, and C~Mao.
\newblock Programming {DNA} self-assembly by geometry.
\newblock {\em J. Am. Chem. Soc.}, 144:8741–8745, 2022.

\end{thebibliography}
\end{document}